\documentclass[12pt,a4paper,reqno]{amsart} %optio titlepage!
\usepackage[T1]{fontenc}       % kirjaimet, joissa aksentteja (skandit)
\usepackage{amsfonts}          % AMS-fontit
\usepackage{amsmath}
\usepackage{amssymb}
\usepackage{overpic}
\usepackage{euscript} % \mathcal tuottaa hyvдnnдkцisen fontin
\usepackage{eurosym}
\usepackage{comment}
\usepackage{mathrsfs} % calligraphic font via the command \mathscr{ABC}
\usepackage{ifthen}
\usepackage{esint} %k.a.integraali komennolla \fint
\usepackage{color}

\long\def\symbolfootnote[#1]#2{\begingroup%
\def\thefootnote{\fnsymbol{footnote}}\footnote[#1]{#2}\endgroup}

{\qed\vspace{5pt}}

\usepackage{amsthm}

\newtheoremstyle{lause}% name
{5pt}% space above
{5pt}% space below
{\slshape}% body font
{\parindent}% indent amount (empty = no indent)
{\bfseries}% theorem head font
{.}% punctuation after theorem head
{.5em}% space after theorem head
{}% theorem head spec (can be left empty, meaning 'normal')

\theoremstyle{lause}
\newtheoremstyle{maaritelma}% name
{5pt}% space above
{5pt}% space below
{\rmfamily}% body font
{\parindent}% indent amount (empty = no indent)
{\bfseries}% theorem head font
{.}% punctuation after theorem head
{.5em}% space after theorem head
{}% theorem head spec (can be left empty, meaning 'normal')

\theoremstyle{maaritelma}
\newtheoremstyle{lause}% name
{5pt}% space above
{5pt}% space below
{\slshape}% body font
{\parindent}% indent amount (empty = no indent)
{\bfseries}% theorem head font
{.}% punctuation after theorem head
{.5em}% space after theorem head
{}% theorem head spec (can be left empty, meaning 'normal')

\theoremstyle{lause}
\newtheorem{theorem}{Theorem}[section]
\newtheorem{lemma}[theorem]{Lemma}

\newtheorem{corollary}[theorem]{Corollary}

\newtheoremstyle{maaritelma}% name
{5pt}% space above
{5pt}% space below
{\rmfamily}% body font
{\parindent}% indent amount (empty = no indent)
{\bfseries}% theorem head font
{.}% punctuation after theorem head
{.5em}% space after theorem head
{}% theorem head spec (can be left empty, meaning 'normal')

\theoremstyle{maaritelma}
\newtheorem{definition}[theorem]{Definition}

\newtheorem{remark}[theorem]{Remark}

\numberwithin{equation}{section}

\begin{document}

\thispagestyle{empty}

\begin{center}

{\large{\textbf{On the theory of capacities on locally compact spaces\\ and its interaction with the theory of balayage}}}

\vspace{18pt}

\textbf{Natalia Zorii}

\vspace{18pt}

\emph{In memory of Makoto Ohtsuka (1922--2007)}\vspace{8pt}

\footnotesize{\address{Institute of Mathematics, Academy of Sciences
of Ukraine, Tereshchenkivska~3, 01601,
Kyiv-4, Ukraine\\
natalia.zorii@gmail.com }}

\end{center}

\vspace{12pt}

{\footnotesize{\textbf{Abstract.}
The paper deals with the theory of inner (outer) capacities on locally compact spaces with respect to general function kernels,
the main emphasis being placed on the establishment of alternative  characterizations of inner (outer) capacities and inner (outer) capacitary measures for arbitrary sets. The analysis is substantially based on the close interaction between the theory of capacities and that of balayage. As a by-product, we provide a rigorous justification of Fuglede's theories of inner and outer capacitary measures and capacitability (Acta Math., 1960).
The results obtained are largely new even for the logarithmic, Newtonian, Green, $\alpha$-Riesz, and $\alpha$-Green kernels.
}}
\symbolfootnote[0]{\quad 2010 Mathematics Subject Classification: Primary 31C15.}
\symbolfootnote[0]{\quad Key words: Inner and outer capacities and capacitary measures, capacitability, inner and outer balayage, perfect kernels, first and second maximum principles, principle of positivity of mass.}

\vspace{6pt}

\markboth{\emph{Natalia Zorii}} {\emph{On the theory of capacities on locally compact spaces}}

\section{Introduction}

The classical theory of capacities of {\it compact\/} sets on $\mathbb R^n$, $n\geqslant2$, was  initiated by N.~Wiener \cite{NW}, and further developed by O.~Frostman \cite{Fr}, M.~Riesz \cite{Riesz}, and C.~de la Val\-l\'{e}e-Pou\-ssin \cite{V}. The main tool is here the {\it vague\/} topology, for the energy is vaguely lower semicontinuous
on positive measures, whereas the set of admissible measures in the principal minimum energy problem is vaguely compact.

The modern theory of inner and outer capacities of {\it arbitrary\/} sets was originated in the pioneering papers by H.~Cartan \cite{Ca1,Ca2}. His approach is mainly based on the discovery that the cone of all positive measures on $\mathbb R^n$, $n\geqslant3$, of finite Newtonian energy is {\it complete in the strong topology}, determined by the energy norm.

The systematic use of both the strong and the vague topologies enabled B.~Fuglede \cite{F1} to develop a theory of inner (outer) capacities and inner (outer) capacitary measures on a locally compact space $X$, generalizing Cartan's theory to quite a large class of positive definite kernels $\kappa$ on $X$ satisfying the principle of {\it consistency}.

Similar ideas were further explored by M.~Ohtsuka \cite{O} for {\it vec\-tor-val\-ued Radon measures\/} $\boldsymbol{\mu}$ on $X$, influenced by external fields; see also relevant papers
by the present author \cite{Z4}--\cite{ZPot3}, where the measures $\boldsymbol{\mu}$ were even allowed to be {\it infinite dimensional}.

In this work we continue the study of inner (outer) capacities and inner (outer) capacitary measures of sets $A\subset X$ for suitable kernels $\kappa$ on $X$, the main emphasis being placed on the establishment of alternative characterizations of these concepts for {\it arbitrary\/} $A$. (In the existing literature on this subject, mainly pertaining to the classical  kernels on $\mathbb R^n$, one usually imposed upon $A$ some topological restrictions, see e.g.\
\cite[p.~138]{AG}, \cite[p.~241--243]{Doob},
\cite[p.~138--139]{L}, \cite{V0,NW}.)

The alternative characterizations of inner capacities and inner capacitary measures thereby obtained are given in particular by Theorems~\ref{th-ap}, \ref{cor52}, \ref{cor3}, \ref{G0'ar}, \ref{th-char}, \ref{cor2alt} and Corollaries~\ref{cor52'}, \ref{cor-char}.
The analysis performed is substantially based on the close interaction between the theory of capacities and that of balayage.
(Concerning the theory of balayage on a locally compact space $X$ with respect to a suitable kernel~$\kappa$, see the author's recent paper \cite{Z-arx1}; see also Sects.~\ref{i1}, \ref{sec-bal-ou} below for further relevant results, in particular Theorems~\ref{th-intr}, \ref{th-bal-ou} for a number of alternative characterizations of inner and outer swept measures.)

We also investigate the convergence of inner capacities, inner capacitary measures, and their potentials
for monotone families of sets (Theorems~\ref{cor2}, \ref{th-cont2}, \ref{th-cont-bor2}).

Based on the theory of
inner capacities and inner capacitary measures thereby developed, we further carry out a similar analysis for outer capacities and outer capacitary measures.
See in particular alternative characterizations of outer capacities and outer capacitary measures described in Theorems~\ref{G0'-ou}--\ref{cor2alt-ou}, as well as convergence results for monotone families of sets presented in Theorems~\ref{cor2-ou} and \ref{th-cont-bor2-ou}.\footnote{In the investigation of these {\it outer\/} concepts we must however impose upon $X$ and $A$ certain topological restrictions, e.g.\ that the space $X$ be second-countable while the sets $A\subset X$ be Borel.}

To provide an example of the results obtained,\footnote{In this work we are only concerned with the characterizations of capacities expressed in potential theoretical terms, leaving aside those given in terms of transfinite diameters, Chebyshev constants, extremal lengths, or Dirichlet integrals of suitable functions (cf.\ e.g.\ \cite{Ch2,FN,F0,ST,O1,Omon,Z0}).}
consider a {\it perfect kernel\/ $\kappa$ satisfying the first and second maximum principles\/} (for definitions see Sect.~\ref{sec-1}), and let $\mathcal E^+$ stand for the cone of all positive (scalar) Radon measures $\mu$ on $X$ of finite energy: \[\kappa(\mu,\mu):=\int\kappa(x,y)\,d(\mu\otimes\mu)(x,y)=\int\kappa\mu\,d\mu<\infty,\]
$\kappa\mu(\cdot):=\int\kappa(\cdot,y)\,d\mu(y)$ being the potential of $\mu$. Given {\it arbitrary\/} $A\subset X$, we denote by $\mathcal E^+_A$ the cone of all $\mu\in\mathcal E^+$ concentrated on $A$, by $\mathcal E'_A$ the closure of $\mathcal E^+_A$ in the strong topology on $\mathcal E^+$, determined by the energy norm $\|\mu\|:=\sqrt{\kappa(\mu,\mu)}$, and by $\widehat{\mathcal E}_A'$ the class of all $\mu\in\mathcal E'_A$ with the property $\kappa\mu\leqslant1$ $\mu$-a.e.\ on $X$. If moreover the set $A$ is of {\it finite\/} inner capacity $c_*(A)$, there exists the unique inner equilibrium measure $\gamma_A$, introduced as the (unique) solution to the minimum energy problem over the class
\[\Gamma^+_A:=\bigl\{\mu\in\mathcal E^+:\ \kappa\mu\geqslant1\text{ \ n.e.\ on $A$}\bigr\},\]
where `n.e.\ on $A$' means that the inequality holds everywhere on $A$ except for $E\subset A$ with $c_*(E)=0$ (see Sect.~\ref{sec-2}).

\begin{theorem}[{\rm see Sect.~\ref{sec-cr}}]\label{th-ex} The inner equilibrium measure\/ $\gamma_A$ can alternatively be characterized by means of any of the following\/ {\rm(}equivalent\/{\rm)} assertions\/ {\rm(a)--(f)}:
\begin{itemize}\item[{\rm(a)}]$\gamma_A$ is the unique measure of minimum potential in the class\/ $\Gamma^+_A$, that is, $\gamma_A\in\Gamma^+_A$ and
\[\kappa\gamma_A=\min_{\mu\in\Gamma^+_A}\kappa\mu\text{ \ on\/ $X$.}\]
\item[{\rm(b)}]$\gamma_A$ is the unique measure maximizing\/ $G(\mu):=2\mu(X)-\|\mu\|^2$ over\/ $\mathcal E_A'$, that is, $\gamma_A\in\mathcal E_A'$ and
\[G(\gamma_A)=\max_{\mu\in\mathcal E'_A}\,G(\mu)\quad\bigl({}=c_*(A)\bigr).\]
\item[{\rm(c)}]$\gamma_A$ is the unique measure maximizing\/ $\mu(X)$ over\/ $\widehat{\mathcal E}_A'$, that is, $\gamma_A\in\widehat{\mathcal E}_A'$ and
\[\gamma_A(X)=\max_{\mu\in\widehat{\mathcal E}'_A}\,\mu(X)\quad\bigl({}=c_*(A)\bigr).\]
\item[{\rm(d)}]$\gamma_A$ is the unique measure in\/ $\mathcal E'_A$ satisfying the two inequalities
    \[\gamma_A(X)\geqslant\|\gamma_A\|^2\geqslant c_*(A).\]
\item[{\rm(e)}]$\gamma_A$ is the unique measure in\/ $\mathcal E'_A$ having the property\/\footnote{It follows that there is in general no $\nu\in\mathcal E^+_A$ with $\kappa\nu=1$ n.e.\ on $A$ (unless of course $\mathcal E^+_A$ is strongly closed, cf.\ Remark~\ref{rem2-intr}). Indeed, let $A:=B_r:=\{|x|<r\}\subset\mathbb R^n$, where $n\geqslant3$ and $r\in(0,\infty)$, and let $\kappa(x,y)$ be the Newtonian kernel $|x-y|^{2-n}$. If there were $\nu\in\mathcal E^+_{B_r}\subset\mathcal E'_{B_r}$ with $\kappa\nu=1$ n.e.\ on $B_r$, then, according to (e), $\nu$ would necessarily serve as $\gamma_{B_r}$. But $\gamma_{B_r}$ is the positive measure of total mass $r^{n-2}$, uniformly distributed over the sphere $\{|x|=r\}$; therefore, $\gamma_{B_r}\not\in\mathcal E^+_{B_r}$. Contradiction.\label{Foot}}
    \[\kappa\gamma_A=1\text{ \ n.e.\ on\/ $A$.}\]
\item[{\rm(f)}]$\gamma_A$ is the unique measure in\/ $\mathcal E^+$ satisfying any of the three limit relations
    \begin{align*}&\gamma_K\to\gamma_A\text{ \ strongly in\/ $\mathcal E^+$},\\
    &\gamma_K\to\gamma_A\text{ \ vaguely in\/ $\mathcal E^+$},\\
    &\kappa\gamma_K\uparrow\kappa\gamma_A\text{ \ pointwise on\/ $X$},\end{align*}
    where\/ $K$ ranges through the upward directed set\/ $\mathfrak C_A$ of all compact subsets of\/ $A$, whereas\/ $\gamma_K$ denotes the only measure in\/ $\mathcal E^+_K$ with\/ $\kappa\gamma_K=1$ n.e.\ on\/ $K$.
\end{itemize}\end{theorem}

\begin{remark}Being valid for {\it arbitrary\/} $A\subset X$, Theorem~\ref{th-ex} seems to be largely new even for the Green kernels associated with the Laplacian on Greenian sets in $\mathbb R^n$, $n\geqslant2$ (in particular, for the Newtonian kernel on $\mathbb R^n$, $n\geqslant3$), as well as for the $\alpha$-Riesz $|x-y|^{\alpha-n}$ and the associated $\alpha$-Green kernels on $\mathbb R^n$, where $\alpha<2\leqslant n$.\end{remark}

\begin{remark}\label{rem2-intr}In the particular case where $\mathcal E^+_A$ is {\it strongly closed\/} (which occurs e.g.\ if the set $A$ is closed or even quasiclosed, cf.\ Lemma~\ref{l-quasi}), (b)--(e) in Theorem~\ref{th-ex} remain valid with $\mathcal E'_A$ replaced by $\mathcal E^+_A$ throughout. Moreover, such refined characterizations of $\gamma_A$ actually hold true even for {\it any perfect\/} kernel $\kappa$, except for (e) where $\kappa$ is additionally required to satisfy Frostman's maximum principle (cf.\ Theorem~\ref{cor3}).\end{remark}

\begin{remark}If moreover $X$ is sec\-ond-count\-able (or, more generally, $\sigma$-compact and perfectly normal) while $A$ is Borel, then Theorem~\ref{th-ex} still holds with the inner concepts replaced by the {\it outer\/} concepts throughout (for details see Sect.~\ref{sec-ou-chara}).\end{remark}

\section{On the theory of potentials on locally compact spaces}\label{sec-1}

In this section we review some basic facts of the theory of potentials on a locally compact (Hausdorff) space $X$, to be used throughout the paper. We denote by $\mathfrak M$
the linear space of all (real-valued scalar Radon) measures $\mu$ on $X$ equipped with the {\it vague\/} topology of pointwise convergence on the class $C_0(X)$
of all continuous functions $f:X\to\mathbb R$ of compact support,
and by $\mathfrak M^+$ the cone of all positive $\mu\in\mathfrak M$.

\begin{lemma}[{\rm see \cite[Section~IV.1, Proposition~4]{B2}}]\label{lemma-semi}For any lower semicontinuous\/ {\rm(}l.s.c.{\rm)}\  function\/ $\psi:X\to[0,\infty]$, the mapping\/ $\mu\mapsto\int\psi\,d\mu$ is
vaguely l.s.c.\ on\/ $\mathfrak M^+$, the integral here being understood as an upper integral.\end{lemma}

For any set $A\subset X$, denote by $\mathfrak M^+_A$ the cone of all $\mu\in\mathfrak M^+$ {\it concentrated on\/}
$A$, which means that $A^c:=X\setminus A$ is locally $\mu$-neg\-lig\-ible, or equivalently that $A$ is $\mu$-meas\-ur\-able and $\mu=\mu|_A$, $\mu|_A:=1_A\cdot\mu$ being the {\it trace\/} of $\mu$ to $A$ \cite[Section~IV.14.7]{E2}. (Note that for $\mu\in\mathfrak M^+_A$, the indicator function $1_A$ of $A$ is locally $\mu$-int\-egr\-able.) The total mass of $\mu\in\mathfrak M^+_A$ is $\mu(X)=\mu_*(A)$, $\mu_*(A)$ and $\mu^*(A)$ denoting the {\it inner\/} and the {\it outer\/} measure of $A$, respectively. If moreover $A$ is closed, or if $A^c$ is contained in a countable union of sets $Q_j$ with $\mu^*(Q_j)<\infty$,\footnote{If the latter holds, $A^c$ is said to be $\mu$-$\sigma$-{\it finite\/} \cite[Section~IV.7.3]{E2}. This in particular occurs if the measure $\mu$ is {\it bounded\/} (that is, with $\mu(X)<\infty$), or if the locally compact space $X$ is {\it $\sigma$-com\-pact\/} (that is, representable as a countable union of compact sets \cite[Section~I.9, Definition~5]{B1}).} then for any $\mu\in\mathfrak M^+_A$, $A^c$ is $\mu$-neg\-lig\-ible, that is, $\mu^*(A^c)=0$. In particular, if $A$ is closed, $\mathfrak M^+_A$ consists of all $\mu\in\mathfrak M^+$ with support $S(\mu)\subset A$, cf.\ \cite[Section~III.2.2]{B2}.

Given $A\subset X$, denote by $\mathfrak C_A$ the upward directed set of all compact subsets $K$ of $A$, where $K_1\leqslant K_2$ if and only if $K_1\subset K_2$. If a net $(x_K)_{K\in\mathfrak C_A}\subset Y$ converges to $x_0\in Y$, $Y$ being a topological space, then we shall indicate this fact by writing
\begin{equation*}\label{abr}x_K\to x_0\text{ \ in $Y$ as $K\uparrow A$}.\end{equation*}

\begin{lemma}[{\rm cf.\ \cite[Lemma~1.2.2]{F1}}]\label{l-lower}For any l.s.c.\ function\/ $\psi:X\to[0,\infty]$, any measure\/ $\mu\in\mathfrak M^+$, and any\/ $\mu$-measurable set\/ $A\subset X$,
\[\int\psi\,d\mu|_A=\lim_{K\uparrow A}\,\int\psi\,d\mu|_K.\]
\end{lemma}

Following Fuglede \cite{F1}, a {\it kernel\/} $\kappa$ on $X$ is meant to be a symmetric, l.s.c.\ function $\kappa:X\times X\to(-\infty,\infty]$ that is positive (${}\geqslant0$) unless the space $X$ is compact.\footnote{The latter case can often be reduced to the former by introducing a kernel $\kappa':=\kappa+q\geqslant0$, where $q\in\mathbb R$, which is always possible, for a l.s.c.\ function on a compact space is lower bounded.\label{fIII}}
Given (signed) measures $\mu,\nu\in\mathfrak M$, define the {\it potential\/} and the {\it mutual energy\/}  by
\begin{align*}\kappa\mu(x)&:=\int\kappa(x,y)\,d\mu(y),\quad x\in X,\\
\kappa(\mu,\nu)&:=\int\kappa(x,y)\,d(\mu\otimes\nu)(x,y),
\end{align*}
respectively, provided the right-hand side is well defined as a finite number or $\pm\infty$ (for more details see \cite[Section~2.1]{F1}). For $\mu=\nu$, $\kappa(\mu,\nu)$ defines the {\it energy\/} $\kappa(\mu,\mu)$ of $\mu$.
In particular, if the measures are positive, then $\kappa\mu(x)$, resp.\ $\kappa(\mu,\nu)$, is always well defined and represents a l.s.c.\ function of $(x,\mu)\in X\times\mathfrak M^+$, resp.\ of $(\mu,\nu)\in\mathfrak M^+\times\mathfrak M^+$ (the {\it principle of descent\/} \cite[Lemma~2.2.1]{F1}, cf.\ Lemma~\ref{lemma-semi} and footnote~\ref{fIII}).

\begin{lemma}\label{222} For any\/ $A\subset X$ and any\/ $\mu\in\mathfrak M^+_A$,
\begin{equation*}\label{emon}\kappa(\mu,\mu)=\lim_{K\uparrow A}\,\kappa(\mu|_K,\mu|_K).\end{equation*}
\end{lemma}

\begin{proof}
 If $\kappa\geqslant0$, this follows from Lemma~\ref{l-lower} noting that $(\mu\otimes\mu)|_{K\times K}=\mu|_K\otimes\mu|_K$, whereas the remaining case of compact $X$ is treated as described in footnote~\ref{fIII}.
\end{proof}

From now on we shall always assume that a kernel $\kappa$ is {\it pseudo-positive\/} \cite[p.~150]{F1}, which means that $\kappa(\mu,\mu)\geqslant0$ for all $\mu\in\mathfrak M^+$ (or equivalently, for all $\mu\in\mathfrak M^+$ of compact support, cf.\ Lemma~\ref{222}). A (pse\-udo-pos\-it\-ive) kernel $\kappa$ is said to be {\it strictly pse\-udo-pos\-itive\/} if $\kappa(\mu,\mu)>0$ for all $\mu\in\mathfrak M^+$, $\mu\ne0$, of compact support.

Denote by $\mathcal E$ the set of all (signed) $\mu\in\mathfrak M$ of finite energy $\kappa(\mu,\mu)$, and by $\mathcal E^+$
its subset consisting of all positive $\mu\in\mathcal E$. Given $A\subset X$, also denote
\[\mathcal E^+_A:=\mathcal E\cap\mathfrak M^+_A,\quad\breve{\mathcal E}^+_A:=\bigl\{\mu\in\mathcal E^+_A: \ \mu(X)=1\bigr\},\]
and\footnote{As usual, the infimum over the empty set is taken to be $+\infty$. We also agree that $1/(+\infty)=0$ and $1/0 = +\infty$.\label{f1}}
\begin{equation}w(A):=\inf_{\mu\in\breve{\mathcal E}^+_A}\,\kappa(\mu,\mu)\in[0,\infty].\label{W}\end{equation}
The value \[c_*(A):=1/w(A)\in[0,\infty]\] is said to be the (Wiener) {\it inner capacity\/} of the set $A$ (with respect to the kernel $\kappa$).

It is seen from Lemma~\ref{222} that $c_*(A)$ would be the same if the admissible measures in (\ref{W})
were required to be of compact support. This in turn implies that\footnote{We write $c(A)$ in place of $c_*(A)$ if $A$ is {\it capacitable}, that is, if $c_*(A)=c^*(A)$, where $c^*(A)$ is the {\it outer\/} capacity of $A$, defined as $\inf\,c_*(D)$, $D$ ranging over all open sets containing $A$. This occurs, for instance, if $A$ is compact \cite[Lemma~2.3.4]{F1} or open.}
\begin{equation}\label{153}c_*(A)=\lim_{K\uparrow A}\,c(K).\end{equation}

\begin{lemma}[{\rm cf.\ \cite[Lemma~2.3.1]{F1}}]\label{l-negl1} For any\/ $A\subset X$, \[c_*(A)=0\iff\mathcal E^+_A=\{0\}\iff\mathcal E^+_K=\{0\}\text{\ for all\/ $K\in\mathfrak C_A$}.\]
\end{lemma}

\begin{lemma}\label{l-negl}Given\/ $\mu\in\mathcal E^+$, let\/ $A\subset X$ be a\/ $\mu$-mea\-sur\-ab\-le and\/ $\mu$-$\sigma$-fi\-ni\-te set with\/ $c_*(A)=0$. Then\/ $A$ is\/ $\mu$-neg\-lig\-ible.\end{lemma}

\begin{proof} As $A$ is $\mu$-mea\-sur\-ab\-le and $\mu$-$\sigma$-fi\-ni\-te, $\mu^*(A)=0$  will follow once we show that $\mu(K)=0$ for every compact $K\subset A$, but this is obvious in view of Lemma~\ref{l-negl1}.
\end{proof}

To avoid trivialities, assume throughout that $\mathcal E^+\ne\{0\}$, or equivalently
\begin{equation}\label{x0}c(X)>0\end{equation}
(cf.\ Lemma~\ref{l-negl1}); then the kernel $\kappa$ is necessarily not identically infinite: $\kappa\not\equiv+\infty$.

An assertion $\mathcal A(x)$ involving a variable point $x\in X$ is said to hold {\it nearly everywhere\/} ({\it n.e.\/}) on $A\subset X$ if $c_*(E)=0$, $E$ being the set of all $x\in A$ where $\mathcal A(x)$ fails. Replacing here $c_*(E)$ by $c^*(E)$ we arrive at the concept of {\it qua\-si-ev\-ery\-whe\-re\/} ({\it q.e.\/}).

Assume for a moment that $A=K\subset X$ is a {\it compact\/} set with $w(K)<\infty$; then the infimum in (\ref{W}) is an actual minimum, the class $\breve{\mathcal E}^+_K$ being vaguely compact. A measure $\lambda_K\in\breve{\mathcal E}^+_K$ of minimal energy $\kappa(\lambda_K,\lambda_K)=w(K)$ is said to be a {\it capacitary distribution of unit mass\/} on $K$. If moreover $w(K)>0$ (which in particular occurs if the kernel is strictly pse\-udo-pos\-it\-ive), one can introduce the (normalized) measure
\[\gamma_K:=\lambda_K/w(K)\in\mathcal E^+_K,\]
called a {\it capacitary measure\/}
on the (compact) set $K$. A capacitary measure $\gamma_K$ is in general {\it non-unique}, and it
has the following properties \cite[Theorem~2.4]{F1}:\footnote{If Frostman's maximum principle holds (Sect.~\ref{subs-pr}), the capacitary measures $\gamma_K$ are also called {\it equilibrium\/} measures on $K$, for their potentials take the constant value $1$ for nearly all $x\in K$. (In fact, it follows from (\ref{eq-pr-c2}) that $\kappa\gamma_K\leqslant1$ on $X$, which together with (\ref{eq-pr-c1}) gives $\kappa\gamma_K=1$ n.e.\ on $K$.)\label{fr}}
\begin{align}
\label{eq-pr-c0}&\gamma_K(X)=\kappa(\gamma_K,\gamma_K)=c(K),\\
\label{eq-pr-c1}&\kappa\gamma_K\geqslant1\text{ \ n.e.\ on\ }K,\\
\label{eq-pr-c2}&\kappa\gamma_K\leqslant1\text{ \ on\ }S(\gamma_K),\\
\label{eq-pr-c3}&\kappa\gamma_K=1\text{ \ $\gamma_K$-a.e.\ on $X$.}
\end{align}

But if $A\subset X$ is {\it noncompact}, the class $\breve{\mathcal E}^+_A$ is no longer vaguely compact, and the infimum in (\ref{W}) is in general not attained among the admissible measures. In other words, for noncompact $A$, minimum energy problem (\ref{W}) is in general {\it unsolvable}.

\subsection{Potential-theoretic principles}\label{subs-pr}
A kernel $\kappa$ is said to be {\it positive definite\/} if the energy $\kappa(\mu,\mu)$ is ${}\geqslant0$ for all $\mu\in\mathfrak M$ (whenever it is well defined); then $\mathcal E$ is a pre-Hil\-bert space with the inner product $\langle\mu,\nu\rangle:=\kappa(\mu,\nu)$ and the energy seminorm $\|\mu\|:=\sqrt{\kappa(\mu,\mu)}$ \cite[Section~3.1]{F1}.\footnote{Any positive definite kernel is of course  pseudo-positive, but not the other way around. Except for Sects.~\ref{sec-pr-pos} and \ref{sec-pse}, in this study we shall mainly focus on positive definite kernels, thereby enjoying the advantages of a pre-Hilbert structure on the space $\mathcal E$.} The topology on $\mathcal E$ determined by this seminorm is said to be {\it strong}. A positive definite kernel is said to satisfy {\it the energy principle\/} (or to be {\it strictly positive definite\/}) if the seminorm $\|\cdot\|$ is a norm. (Concerning various poten\-tial-theore\-tic principles, see e.g.\ the comprehensive work by Ohtsuka \cite{O}.)

If a kernel $\kappa$ is positive definite, the potential of any $\mu\in\mathcal E$ is (well defined and) finite qua\-si-ev\-ery\-whe\-re on $X$ \cite[p.~165, Corollary]{F1}; furthermore, the potentials of any two given $\mu,\nu\in\mathcal E$ are equal nearly everywhere on $X$ if and only if $\|\mu-\nu\|=0$ \cite[Lemma~3.2.1]{F1}. Therefore, for a {\it strictly\/} positive definite kernel $\kappa$,
\begin{equation}\label{E}\kappa\mu=\kappa\nu\text{ \ n.e.\ on $X$}\iff\mu=\nu.\end{equation}

A positive definite kernel $\kappa$ is said to be {\it consistent\/} if every strong Cauchy net $(\mu_s)\subset\mathcal E^+$ converges strongly to any of its vague limit points $\mu$ \cite[Section~3.3]{F1}.\footnote{Since the space $\mathfrak M$ equipped with the vague topology does not necessarily satisfy the first axiom of countability, the vague convergence cannot in general be described in terms of sequences. We follow Moore and Smith's theory of convergence
\cite{MSm}, based on the concept of {\it nets}. However, if a locally compact space $X$ is sec\-ond-count\-able, then the space $\mathfrak M$ is first-count\-able \cite[Lemma~4.5]{Z-arx}, and the use of nets in $\mathfrak M$ may often be avoided.} (Such $\mu$ do exist e.g.\ if the kernel $\kappa$ is strictly positive definite \cite[Lemma~2.5.1]{F1}.)

A positive definite kernel $\kappa$ is said to be {\it perfect\/} if the energy and consistency principles are both fulfilled \cite[Section~3.3]{F1}. It follows that for a perfect kernel $\kappa$, the cone $\mathcal E^+$ is  complete in the induced strong topology, and moreover {\it every strong Cauchy net\/ $(\mu_s)\subset\mathcal E^+$ converges to the same unique limit\/ $\mu_0\in\mathcal E^+$ both strongly and vaguely}. In fact, all the vague limit points of such $(\mu_s)$ must coincide, being equal to the (unique) strong limit $\mu_0$; since the vague topology on $\mathfrak M$ is Hausdorff, the same $\mu_0$ must also serve as the vague limit of $(\mu_s)$ (see \cite[Section~I.9.1, Corollary]{B1}).

The following two maximum principles will often be used in the sequel.
A kernel $\kappa$ is said to satisfy {\it Frostman's maximum principle\/} ({}=\,the {\it first maximum principle\/}) if for any $\mu\in\mathfrak M^+$ with $\kappa\mu\leqslant1$ on $S(\mu)$, the same inequality holds on all of $X$;\footnote{For the purposes of the present paper, it is enough to
require this only for $\mu$ of finite energy.}
and it is said to satisfy the {\it domination principle\/} ({}=\,the {\it second maximum principle\/}) if for any $\mu,\nu\in\mathcal E^+$ with $\kappa\mu\leqslant\kappa\nu$ $\mu$-a.e., the same inequality is fulfilled everywhere on $X$.

\subsection{Principle of positivity of mass}\label{sec-pr-pos}
The principle of positivity of mass goes back to J.~Deny (see e.g.\ \cite{D2}). In the generality stated in Theorem~\ref{pr-pos}, it seems to be new.

\begin{theorem}\label{pr-pos} Assume that a locally compact space\/ $X$ is\/ $\sigma$-compact, a kernel\/ $\kappa$ is strictly pseudo-positive, and the first and second maximum principles are fulfilled.
For any\/ $\mu,\nu\in\mathcal E^+$ with\/
$\kappa\mu\leqslant\kappa\nu$ n.e.\ on\/ $X$, then
\begin{equation}\label{prpos}\mu(X)\leqslant\nu(X).\end{equation}
\end{theorem}

\begin{proof}
Assume first that $\kappa\geqslant0$. The space $X$ being $\sigma$-compact, there is an increasing sequence $(K_j)$ of compact sets with the union $X$. The kernel being strictly pseu\-do-pos\-it\-ive, for each $j$ large enough, we have $0<c(K_j)<\infty$ (cf.\ (\ref{153}) and (\ref{x0})), and there is therefore an equilibrium measure $\gamma_{K_j}$ on the (compact) set $K_j$ (cf.\ footnote~\ref{fr}). By the countable subadditivity of inner capacity on universally measurable sets \cite[Lemma~2.3.5]{F1}, $\kappa\gamma_{K_j}=1=\kappa\gamma_{K_{j+1}}$ n.e.\ on $K_j$, hence $\gamma_{K_j}$-a.e.\ (Lemma~\ref{l-negl}), and consequently $\kappa\gamma_{K_j}\leqslant\kappa\gamma_{K_{j+1}}$ on all of $X$, by the domination principle.
Thus the sequence $(\kappa\gamma_{K_j})$ increases pointwise on $X$ to some function $g\geqslant0$, and moreover
\begin{equation}\label{prposs}g=1\text{ \ n.e.\ on $X$},\end{equation}
again by the countable subadditivity of inner capacity on universally measurable sets.

Given $\mu,\nu\in\mathcal E^+$ with $\kappa\mu\leqslant\kappa\nu$ n.e.\ on $X$, the same inequality is fulfilled $\gamma_{K_j}$-a.e.\ (Lemma~\ref{l-negl}); therefore, by Fubini's theorem,
\[\int\kappa\gamma_{K_j}\,d\mu=\int\kappa\mu\,d\gamma_{K_j}\leqslant
\int\kappa\nu\,d\gamma_{K_j}=\int\kappa\gamma_{K_j}\,d\nu.\]
By the monotone convergence theorem %\cite[Section~IV.1, Theorem~3]{B2}
applied to the increasing sequence $(\kappa\gamma_{K_j})$ of positive functions with the upper envelope $g$ (cf.\ (\ref{prposs})), we obtain $\int g\,d\mu\leqslant\int g\,d\nu$,
whence (\ref{prpos}); for, $g=1$ $\mu$-a.e.\ as well as $\nu$-a.e.\ (Lemma~\ref{l-negl}), $X$ being $\sigma$-compact.

It remains to consider the case of compact $X$; then $0<c(X)<\infty$, and there is therefore an equilibrium measure $\gamma_X$ on the whole of $X$. Moreover, $\kappa\gamma_X=1$ holds true n.e.\ on $X$, hence $\mu$-a.e.\ as well as $\nu$-a.e.\ (Lemma~\ref{l-negl}). Likewise, $\kappa\mu\leqslant\kappa\nu$ $\gamma_X$-a.e., the measure $\gamma_X$ being of finite energy. This altogether gives
\[\mu(X)=\int\kappa\gamma_X\,d\mu=\int\kappa\mu\,d\gamma_X\leqslant\int\kappa\nu\,d\gamma_X=
\int\kappa\gamma_X\,d\nu=\nu(X)\]
as desired.
\end{proof}

\begin{remark}It is seen from the above proof that in the case of compact $X$, the domination principle is unnecessary for the validity of Theorem~\ref{pr-pos}.\label{prposm}\end{remark}

\subsection{Examples}\label{rem:clas} The $\alpha$-Riesz kernels $|x-y|^{\alpha-n}$ of order $\alpha\in(0,2]$, $\alpha<n$, on $\mathbb R^n$, $n\geqslant2$ (thus in particular the Newtonian kernel $|x-y|^{2-n}$ on $\mathbb R^n$, $n\geqslant3$), satisfy the energy and consistency principles as well as the first and second maximum principles \cite[Theorems~1.10, 1.15, 1.18, 1.27, 1.29]{L}. So do the associated $\alpha$-Green kernels on an arbitrary open subset of $\mathbb R^n$, $n\geqslant2$ \cite[Theorems~4.6, 4.9, 4.11]{FZ}. The ($2$-)Green kernel on a planar Greenian set is likewise strictly positive definite \cite[Section~I.XIII.7]{Doob} and consistent \cite{E}, and it fulfills the first and second maximum principles (see \cite[Theorem~5.1.11]{AG}, \cite[Section~I.V.10]{Doob}). Furthermore, all these kernels also satisfy the principle of positivity of mass (Theorem~\ref{pr-pos}).

Finally, the restriction of the logarithmic kernel $-\log\,|x-y|$ to a closed disc in $\mathbb R^2$ of radius ${}<1$ is strictly positive definite and fulfills Frostman's maximum principle \cite[Theorems~1.6, 1.16]{L}; hence, it is perfect \cite[Theorem~3.4.2]{F1}, and satisfies the principle of positivity of mass (cf.\ Theorem~\ref{pr-pos} and Remark~\ref{prposm}). However, the domination principle then fails in general; it does hold only in a weaker sense where the measures $\mu,\nu$ involved in the ab\-ove-quo\-ted definition (see Sect.~\ref{subs-pr}) meet the additional requirement that $\nu(\mathbb R^2)\leqslant\mu(\mathbb R^2)$, cf.\ \cite[Theorem~3.2]{ST}.

\section{Some basic facts of the theory of inner balayage}\label{i1}

In what follows, we shall use the notions and notations introduced in Sect.~\ref{sec-1}.

The theory of inner balayage of $\mu\in\mathcal E^+$ to arbitrary $A\subset X$, generalizing Cartan's theory \cite{Ca2} of Newtonian balayage to a suitable kernel $\kappa$ on a locally compact space $X$, was developed in \cite[Sections~4--8]{Z-arx1}. The present section reviews some basic facts of this theory, to be used in the sequel; throughout it, we require {\it the energy, consistency, and domination principles\/} to be fulfilled.

Assume for a moment that $A=K\subset X$ is {\it compact}. Based on the abo\-ve-men\-tio\-ned principles, one can prove by generalizing the classical Gauss variational method (see \cite{C0,Ca2}, cf.\ also \cite[Section~IV.5.23]{L}) that for any given $\mu\in\mathcal E^+$, there exists $\mu^K\in\mathcal E^+_K$ uniquely determined within $\mathcal E^+_K$ by the equality $\kappa\mu^K=\kappa\mu$ n.e.\ on $K$.\footnote{Concerning the uniqueness, see (\ref{E}) applied to the space $X:=K$ and the kernel $\kappa':=\kappa|_{K\times K}$.}
This $\mu^K$ is said to be the {\it balayage\/}  of $\mu\in\mathcal E^+$ onto (compact) $K$.

If now $A$ is {\it arbitrary}, there is in general no $\mu^A\in\mathcal E^+_A$ having  the property
$\kappa\mu^A=\kappa\mu$ n.e.\ on $A$.\footnote{Compare with Theorem~\ref{th-intr}(c).} Nevertheless, a substantial theory of inner balayage of $\mu\in\mathcal E^+$ to arbitrary $A$ was developed \cite{Z-arx1}, and this was performed by means of several alternative approaches described in Theorem~\ref{th-intr} below.

 To make an exposition of the theory of inner balayage similar to that of inner capacitary measures, cf.\ \cite[Section~4.1]{F1}, we have however chosen here to start with another approach, different from those utilized in \cite{Z-arx1}.

Given $\mu\in\mathcal E^+$ and $A\subset X$, denote
\begin{equation}\label{io-n}\Gamma_{A,\mu}^+:=\bigl\{\nu\in\mathcal E^+: \ \kappa\nu\geqslant\kappa\mu\text{ \ n.e.\ on $A$}\bigr\}.\end{equation}

\begin{definition}\label{def-bal-n} The {\it inner balayage\/} $\mu^A$ of $\mu\in\mathcal E^+$ to $A\subset X$ is defined as the measure of minimum energy in the class $\Gamma_{A,\mu}^+$, that is,  $\mu^A\in\Gamma_{A,\mu}^+$ and\footnote{Minimum energy problem (\ref{e-d}) makes sense, for $\Gamma_{A,\mu}^+\ne\varnothing$ because of the relation $\mu\in\Gamma_{A,\mu}^+$.}
\begin{equation}\label{e-d}\|\mu^A\|^2=\min_{\nu\in\Gamma_{A,\mu}^+}\,\|\nu\|^2.\end{equation}
\end{definition}

\begin{lemma}\label{def-bal-n-u}
  The inner balayage\/ $\mu^A$ is unique\/ {\rm(}if it exists\/{\rm)}.
\end{lemma}

\begin{proof} It is clear from a strengthened version of countable subadditivity for inner capacity,  presented by Lemma~\ref{str}, that the class $\Gamma_{A,\mu}^+$ is {\it convex}. Therefore, the lemma follows by standard arguments based on a pre-Hil\-bert structure on the space~$\mathcal E$.\footnote{In fact, if a set $\Gamma\subset\mathcal E$ is convex, and if $\gamma,\gamma'\in\Gamma$ are two measures of minimum energy, then an application of the parallelogram identity in the pre-Hilbert space $\mathcal E$ gives
\[0\leqslant\|\gamma-\gamma'\|^2=2\|\gamma\|^2+2\|\gamma'\|^2-\|\gamma+\gamma'\|^2\leqslant0,\]
for $(\gamma+\gamma')/2\in\Gamma$. Hence $\gamma=\gamma'$, by the energy principle.\label{unique}}\end{proof}

\begin{lemma}\label{str}
For arbitrary\/ $A\subset X$ and universally measurable\/ $U_j\subset X$, $j\in\mathbb N$,
\[c_*\Bigl(\bigcup_{j\in\mathbb N}\,A\cap U_j\Bigr)\leqslant\sum_{j\in\mathbb N}\,c_*(A\cap U_j).\]\end{lemma}

\begin{proof}See \cite[p.~158, Remark]{F1}, compare with \cite[Lemma~2.3.5]{F1}.\footnote{For the Newtonian kernel $|x-y|^{2-n}$ on $\mathbb R^n$, $n\geqslant3$, this goes back to Cartan \cite[p.~253]{Ca2}. Also note that Lemma~\ref{str} actually holds true for any pse\-udo-pos\-it\-ive kernel on a locally compact space.}\end{proof}

\subsection{On the existence of $\mu^A$ and its alternative characterizations}
As before, we denote by $\mathcal E^+_A$ the convex cone of all $\nu\in\mathcal E^+$ concentrated on $A$, and by $\mathcal E_A'$ the closure of $\mathcal E^+_A$ in the strong topology on $\mathcal E^+$.

\begin{theorem}\label{th-intr}For any\/ $\mu\in\mathcal E^+$ and any\/ $A\subset X$, the inner balayage\/ $\mu^A$, introduced by Definition\/~{\rm\ref{def-bal-n}}, does exist, and it  satisfies the three relations
\begin{align}
\label{eq-pr-10}\kappa\mu^A&=\kappa\mu\text{ \ n.e.\ on\ }A,\\
\label{eq-pr-101}\kappa\mu^A&=\kappa\mu\text{ \ $\mu^A$-a.e.,}\\
\label{eq-pr-102}\kappa\mu^A&\leqslant\kappa\mu\text{ \ on $X$.}
\end{align}
Furthermore, $\mu^A$ can alternatively be characterized by means of any of the following\/ {\rm(}equivalent\/{\rm)} assertions\/ {\rm(a)--(d)}:
\begin{itemize}
\item[\rm(a)] $\mu^A$ is the unique measure of minimum potential in the class\/ $\Gamma_{A,\mu}^+$, that is,  $\mu^A\in\Gamma_{A,\mu}^+$ and
\[\kappa\mu^A=\min_{\nu\in\Gamma_{A,\mu}^+}\,\kappa\nu\text{ \ on\/ $X$}.\]
\item[\rm(b)] $\mu^A$ is the unique orthogonal projection of\/ $\mu$ in the pre-Hilbert space\/ $\mathcal E$ onto the\/ {\rm(}convex, strongly complete\/{\rm)} cone $\mathcal E'_A$, that is, $\mu^A\in\mathcal E'_A$ and\/\footnote{Concerning the concept of orthogonal projection, see e.g.\ \cite[Theorem~1.12.3]{E2}. (Observe that, being a strongly closed subset of the strongly complete cone $\mathcal E^+$, $\mathcal E'_A$ is likewise strongly complete.)}
\[\|\mu-\mu^A\|=\min_{\nu\in\mathcal E_A'}\,\|\mu-\nu\|.\]
\item[\rm(c)] $\mu^A$ is the unique measure in\/ $\mathcal E'_A$ having the property
\[\kappa\mu^A=\kappa\mu\text{ \ n.e.\ on\ }A.\]
\item[\rm(d)] $\mu^A$ is the unique measure in\/ $\mathcal E^+$ satisfying any of the three limit relations
\begin{align*}
 &\mu^K\to\mu^A\text{ \ strongly in\/ $\mathcal E^+$ as\/ $K\uparrow A$},\\
 &\mu^K\to\mu^A\text{ \ vaguely in\/ $\mathcal E^+$ as\/ $K\uparrow A$},\\
 &\kappa\mu^K\uparrow\kappa\mu^A\text{ \ pointwise on\/ $X$ as\/ $K\uparrow A$},
\end{align*}
where\/ $\mu^K$ denotes the only measure in\/ $\mathcal E^+_K$ with\/ $\kappa\mu^K=\kappa\mu$ n.e.\ on\/ $K$.
\end{itemize}
\end{theorem}

\begin{proof} It was shown in \cite[Section~4]{Z-arx1} that there exists $\mu'\in\mathcal E^+$ uniquely determined by any of (a)--(d) (with $\mu^A$ replaced by $\mu'$), and that it satisfies relations (\ref{eq-pr-10})--(\ref{eq-pr-102}) (again with $\mu^A$ replaced by $\mu'$).\footnote{A kernel $\kappa$ was defined in \cite{Z-arx1} to be positive; however, this restriction on $\kappa$ is unnecessary for the validity of the ab\-ove-quo\-ted results from \cite{Z-arx1}.}

It thus remains to show that the same $\mu'$ meets Definition~\ref{def-bal-n}, and so it serves as the (unique) inner balayage $\mu^A$. To this end, one can assume the two conditions
\begin{equation}\label{2}c_*(A)>0\text{ \ and \ }\kappa\mu\ne0\text{ \ n.e.\ on $A$}\end{equation}
to be fulfilled, for if not, zero measure belongs to $\Gamma_{A,\mu}^+$, and therefore
\[\mu^A=0=\mu',\] the former equality being obvious from Definition~\ref{def-bal-n}, while the latter being clear from (a).
Applying (c), we now conclude from (\ref{2}) that then necessarily
\begin{equation}\label{prime}\mu'\ne0.\end{equation}

As $\mu'\in\Gamma_{A,\mu}^+$ by (\ref{eq-pr-10}), the theorem will follow once we show that
\begin{equation}\label{bal-ch1}\|\mu'\|\leqslant\|\nu\|,\end{equation}
$\nu\in\Gamma_{A,\mu}^+$ being arbitrarily given. But, by virtue of (a), $\kappa\mu'\leqslant\kappa\nu$  everywhere on $X$,
which yields, by the Cauchy--Schwarz (Bunyakovski) inequality,
\[\|\mu'\|^2=\int\kappa\mu'\,d\mu'\leqslant\int\kappa\nu\,d\mu'=\langle\nu,\mu'\rangle\leqslant\|\nu\|\cdot\|\mu'\|,\]
and dividing by $\|\mu'\|\ne0$ (cf.\ (\ref{prime})) results in (\ref{bal-ch1}).
\end{proof}

\begin{corollary}\label{bal-tot-m} Assume that\/ $X$ is\/ $\sigma$-compact while\/ $\kappa$ satisfies Frostman's maximum principle.
For any\/ $\mu\in\mathcal E^+$ and\/ $A\subset X$, the inner balayage\/ $\mu^A$ is of minimum total mass in the class\/ $\Gamma_{A,\mu}^+$, that is,
\begin{equation}\label{eq-t-m}\mu^A(X)=\min_{\nu\in\Gamma_{A,\mu}^+}\,\nu(X).\end{equation}
\end{corollary}

\begin{proof}
  Since $\mu^A\in\Gamma_{A,\mu}^+$, we only need to show that
  $\mu^A(X)\leqslant\nu(X)$  for all $\nu\in\Gamma_{A,\mu}^+$, but this follows immediately from Theorem~\ref{th-intr}(a) by use of
  the principle of positivity of mass (Theorem~\ref{pr-pos}).
\end{proof}

\begin{remark}\label{t-m-nonun} However, extremal property (\ref{eq-t-m}) cannot serve as an alternative characterization of inner balayage, for it does not determine $\mu^A$ uniquely within $\Gamma_{A,\mu}^+$. Indeed, consider the $\alpha$-Riesz kernel $|x-y|^{\alpha-n}$ of order $\alpha\leqslant2$, $\alpha<n$, on $\mathbb R^n$, $n\geqslant2$, and a proper, closed subset $A$ of $\mathbb R^n$ that is {\it not\/ $\alpha$-thin at infinity\/} \cite[Definition~3.1]{KM} (take, for instance, $A:=\{|x|\geqslant 1\}$). Then for any $\mu\in\mathcal E^+_{A^c}$,
\begin{equation}\label{eq-t-m1}\mu^A\ne\mu\text{ \ and \ }\mu^A(\mathbb R^n)=\mu(\mathbb R^n),\end{equation}
the former being obvious e.g.\ from Theorem~\ref{th-intr}(b), whereas the latter holds true by   \cite[Corollary~5.3]{Z-bal2}.
Noting that $\mu,\mu^A\in\Gamma_{A,\mu}^+$, we conclude by combining (\ref{eq-t-m}) with (\ref{eq-t-m1}) that there are actually infinitely many measures of minimum total mass in $\Gamma_{A,\mu}^+$, for so is every measure of the form $a\mu+b\mu^A$, where $a,b\in[0,1]$ and $a+b=1$.
\end{remark}

\subsection{The case of the strongly closed $\mathcal E^+_A$} In the particular case where the class $\mathcal E^+_A$ is {\it strongly closed}, assertions (b) and (c) in Theorem~\ref{th-intr} remain valid with $\mathcal E^+_A$ in place of $\mathcal E'_A$, and these refined characterizations of inner balayage do hold, for instance, if the set $A$ is closed or even quasiclosed (see Definition~\ref{def-quasi} and Lemma~\ref{l-quasi}).

\begin{definition}[{\rm Fuglede \cite{F71}}]\label{def-quasi}
A set $A\subset X$ is said to be {\it quasiclosed\/} if
it can be approximated in outer capacity by closed sets, that is, if
\begin{equation*}\label{def-q}
\inf\,\bigl\{c^*(A\bigtriangleup F):\ F\text{ closed, }F\subset X\bigr\}=0,
\end{equation*}
$\bigtriangleup$ being the symmetric difference.
\end{definition}

\begin{lemma}\label{l-quasi}If\/ $A\subset X$ is quasiclosed, the cone\/ $\mathcal E^+_A$ is strongly closed; hence
\[\mathcal E'_A=\mathcal E^+_A.\]
\end{lemma}
\begin{proof}
The kernel $\kappa$ being perfect, the strong convergence of any net $(\nu_s)\subset\mathcal E^+_A$ to $\nu_0\in\mathcal E^+$ implies the vague convergence to the same limit. Since for quasiclosed $A$,  $\mathcal E^+_A$ is vaguely closed in consequence of \cite[Corollary~6.2]{Fu4}, the lemma follows.
\end{proof}

\section{Characterizations of inner capacity for pse\-udo-pos\-itive kernels}\label{sec-pse}

In this section, $\kappa$ is {\it pse\-udo-pos\-itive\/} (Sect.~\ref{sec-1}). For any $\mu\in\mathcal E^+$ and $A\subset X$, denote
\begin{align}
&G(\mu):=2\mu(X)-\kappa(\mu,\mu)\in(-\infty,\infty],\label{def1}\\
&\widehat{\mathcal E}^+_A:=\bigl\{\nu\in\mathcal E^+_A: \ \kappa\nu\leqslant1\text{\ $\nu$-a.e.\ on $X$}\bigr\}.\label{def2}
\end{align}

\begin{theorem}\label{th-ap}
For any set\/ $A\subset X$ and any pse\-udo-pos\-itive kernel\/ $\kappa$,
\begin{equation}\label{G0}
c_*(A)=\sup_{\nu\in\mathcal E^+_A}\,G(\nu).\end{equation}
\end{theorem}

\begin{remark}\label{Fuu}Formula (\ref{G0}) was declared by Fuglede \cite[p.~162, Eq.~(1)]{F1}; however, the proof given in \cite[p.~162, footnote~1]{F1} was incomplete, being actually only applicable to a relatively compact set $A$ and a strictly pse\-udo-pos\-itive kernel $\kappa$. This in turn has led to the incompleteness of some other proofs in \cite{F1}; thus the theories of inner and outer capacitary measures and capacitability \cite[Sections~4.2--4.5]{F1} turned out to be unjustified. In the present work we have restored Fuglede's theories in full generality by showing that formula (\ref{G0}) does indeed hold as stated (see Sect.~\ref{proof}).\end{remark}

\begin{theorem}\label{cor52}Given\/ $A\subset X$ and a pse\-udo-pos\-itive kernel\/ $\kappa$, assume that either\/ $\kappa\geqslant0$ or\/ $c(K)<\infty$ for all compact\/ $K\subset A$.\footnote{This occurs e.g.\ if either $X$ is noncompact, or $\kappa$ is strictly pse\-udo-pos\-itive, or $c_*(A)<\infty$.} Then
\begin{equation}\label{G00}c_*(A)=\sup_{\nu\in\widehat{\mathcal E}^+_A}\,\nu(X).\end{equation}
\end{theorem}

\begin{corollary}\label{cor52'}For any pse\-udo-pos\-itive kernel\/ $\kappa$ and any compact set\/ $K\subset X$ with\/ $c(K)\in(0,\infty)$, extremal problems\/ {\rm(\ref{G0})} and\/ {\rm(\ref{G00})} are both solvable with the same maximizing measures, forming the class of capacitary measures\/ $\gamma_K$ on\/ $K$.\end{corollary}

\begin{proof}
  This is clear from (\ref{G0comp}), (\ref{mcompp'''}), and (\ref{mcompp''}) (applied to compact $A:=K$).
\end{proof}

\begin{remark} Let now $A\subset X$ be {\it noncompact}, and $c_*(A)<\infty$. We shall show below that if moreover the kernel $\kappa$ is {\it perfect}, whereas the class $\mathcal E^+_A$ is {\it strongly closed\/} (or in particular if $A$ is quasiclosed), then each of extremal problems (\ref{G0}) and (\ref{G00}) is uniquely solvable with the same maximizing measure, serving simultaneously as the unique inner capacitary measure $\gamma_A$ (Theorem~\ref{cor3}). This result admits a further generalization to {\it arbitrary\/} $A$, which however additionally requires upon the perfect kernel $\kappa$ {\it the first and the second maximum principles\/} (see Theorem~\ref{G0'ar}).\end{remark}

\subsection{Proof of Theorem~\ref{th-ap}}\label{proof} Assume $c_*(A)>0$, since otherwise $\mathcal E^+_A=\{0\}$ by Lemma~\ref{l-negl1}, and (\ref{G0}) is trivial. There is also no loss of generality in assuming that
\begin{equation}\label{b1}\kappa(\nu,\nu)>0\text{ \ for all bounded\ }\nu\in\mathcal E^+_A, \ \nu\ne0,\end{equation}
for if not, both sides in (\ref{G0}) are equal to $+\infty$. In fact, if there exists $\xi\in\mathcal E^+_A$, $\xi\ne0$, with $\xi(X)<\infty$ and $\kappa(\xi,\xi)=0$, then $c_*(A)=\infty$, for $\xi':=\xi/\xi(X)\in\breve{\mathcal E}^+_A$, whence
\[0\leqslant w(A)\leqslant\kappa(\xi',\xi')=0.\] Noting also that $h\xi\in\mathcal E^+_A$ for all $h>0$, whereas $G(h\xi)=2h\xi(X)>0$, we arrive at (\ref{G0}) by letting $h\to\infty$.

Suppose first that the space $X$ is compact; then all the measures on $X$ must be {\it bounded}. Excluding $\nu=0$, for every given $\nu\in\mathcal E^+_A$ write $\nu=q\mu$, where $\mu\in\breve{\mathcal E}^+_A$ and $q:=\nu(X)\in(0,\infty)$. Then $G(\nu)$ takes the form
$2q-q^2\kappa(\mu,\mu)$,
which attains its maximum in $q\in(0,\infty)$ at $q_0:=1/\kappa(\mu,\mu)\in(0,\infty)$, cf.\ (\ref{b1}), and this maximum equals $q_0=1/\kappa(\mu,\mu)\in(0,\infty)$. Taking now the supremum over $\mu\in\breve{\mathcal E}^+_A$ we obtain
\begin{equation}\label{G0comp}
\sup_{\nu\in\mathcal E^+_A}\,G(\nu)=\Bigl[\inf_{\mu\in\breve{\mathcal E}^+_A}\,\kappa(\mu,\mu)\Bigr]^{-1}=c_*(A).
\end{equation}

To complete the proof, it remains to consider noncompact $X$; then $\kappa\geqslant0$ by convention.
By Lemma~\ref{l-lower} applied to $\psi:=\kappa$, and subsequently to $\psi:=1$,
\begin{equation}\label{L}
 G(\nu)=\lim_{K\uparrow A}\,G(\nu|_K)\text{ \ for all $\nu\in\mathcal E^+_A$}.
\end{equation}
Therefore,
\[\sup_{\nu\in\mathcal E^+_A}\,G(\nu)=\lim_{K\uparrow A}\,\Bigl(\sup_{\nu\in\mathcal E^+_K}\,G(\nu)\Bigr)=\lim_{K\uparrow A}\,c(K)=c_*(A),\]
where the first equality follows easily from (\ref{L}), while the second holds by (\ref{G0comp}), applied to compact $X:=K$ and the kernel $\kappa|_{K\times K}$.  Formula (\ref{G0}) is thus proved.

\subsection{Proof of Theorem~\ref{cor52}}\label{prooff}
Likewise as in the preceding proof, assume $c_*(A)>0$, since otherwise $\widehat{\mathcal E}^+_A=\mathcal E^+_A=\{0\}$ (Lemma~\ref{l-negl1}), and (\ref{G00}) is trivial.

In the case where $\kappa\geqslant0$, one can also assume (\ref{b1}) to be fulfilled,
for if not, both sides in (\ref{G00}) equal $+\infty$. In fact, if there exists a bounded measure $\xi\in\mathcal E^+_A$, $\xi\ne0$, with $\kappa(\xi,\xi)=0$, then $c_*(A)=\infty$ (see Sect.~\ref{proof}). On the other hand, since $\kappa\xi\geqslant0$ on $X$, we infer from $\kappa(\xi,\xi)=\int\kappa\xi\,d\xi=0$ that $\kappa\xi=0$ $\xi$-a.e. This in turn implies that $h\xi\in\widehat{\mathcal E}^+_A$ for all $h>0$, which leads to (\ref{G00}) by taking $h$ infinitely large.

We stress that  under condition (\ref{b1}), the restriction of $\kappa$ to $K\times K$, $K\subset A$ being compact, represents a strictly pse\-udo-pos\-itive kernel on the (compact) space $K$.

It follows that under either of the hypotheses of the theorem, for each $K\in\mathfrak C_A$ sufficiently large, we have $0<c(K)<\infty$, and hence a capacitary measure $\gamma_K$ on $K$ does exist. Since $\gamma_K(X)=c(K)$ while $\gamma_K\in\widehat{\mathcal E}^+_K$ (cf.\ (\ref{eq-pr-c0}) and (\ref{eq-pr-c3})),
\begin{equation}\label{mcompp'''}c(K)=\gamma_K(X)\leqslant\sup_{\nu\in\widehat{\mathcal E}^+_K}\,\nu(X)\leqslant\sup_{\nu\in\widehat{\mathcal E}^+_A}\,\nu(X).\end{equation}

On the other hand, it is clear from (\ref{def1}) and (\ref{def2}) that for every $\nu\in\widehat{\mathcal E}^+_A$,
\begin{equation}\label{estt}
G(\nu)\geqslant\nu(X),\end{equation}
because $\kappa(\nu,\nu)=\int\kappa\nu\,d\nu\leqslant\nu(X)$. On account of (\ref{G0}), we therefore get
\begin{equation}\label{mcompp''}
  \sup_{\nu\in\widehat{\mathcal E}^+_A}\,\nu(X)\leqslant\sup_{\nu\in\widehat{\mathcal E}^+_A}\,G(\nu)\leqslant\sup_{\nu\in\mathcal E^+_A}\,G(\nu)=c_*(A).
\end{equation}
Combining this with (\ref{mcompp'''}), and then letting $K\uparrow A$ results in (\ref{G00}).

\section{An alternative proof of Theorem~\ref{th-ap}}\label{sec-further}

Throughout this section, a kernel $\kappa$ is assumed to be {\it strictly positive definite}.
Then Theorem~\ref{th-ap} admits an alternative proof, to be particularly useful in the sequel.
It is based on the following crucial auxiliary result, being of independent interest.

\begin{theorem}\label{capf} For any\/ $A\subset X$,
the following two assertions are equivalent:
\begin{itemize}
   \item[{\rm(i)}] $c_*(A)<\infty$.
   \item[{\rm(ii)}] $\nu(X)<\infty$ for all\/ $\nu\in\mathcal E^+_A$.
   \end{itemize}
\end{theorem}

\begin{remark}\label{pseudo}
 Actually, (i)$\Rightarrow$(ii) holds true even for any pseu\-do-pos\-itive kernel. This follows in the same manner as in the proof below, the only difference being in replacing $w(K)>0$ by $w(K)\geqslant0$, see (\ref{th1111}).
\end{remark}

\begin{remark}\label{pseudo-no}
 On the contrary, (ii)$\Rightarrow$(i) would not hold in general if the requirement of the energy principle were omitted. (Take $X$ to be compact, $\kappa\equiv0$, $A:=X$. Then $c_*(A)=\infty$, though the measures on the compact space $X$ must be bounded.)
\end{remark}

\subsection{Proof of Theorem~\ref{capf}}\label{subs}
To prove (i)$\Rightarrow$(ii), let to the contrary there be $\nu\in\mathcal E^+_A$ with $\nu(X)=\infty$.
Then
\begin{equation}\label{th0}\lim_{K\uparrow A}\,\nu(K)=\nu_*(A)=\nu(X)=\infty,\end{equation}
and hence one can choose $K_0\in\mathfrak C_A$ so that $\nu(K)>0$ for all $K\geqslant K_0$ ($K\in\mathfrak C_A$).
Since for each of those $K$, the measure $\nu|_K/\nu(K)$ belongs to $\breve{\mathcal E}^+_K$,
\begin{equation}\label{th1111}\frac{\|\nu|_K\|^2}{\nu(K)^2}\geqslant w(K)>0.\end{equation}
Letting now $K\uparrow A$ we obtain $w(A)=0$ by virtue of (\ref{153}), (\ref{th0}), and Lemma~\ref{222}.
Thus $c_*(A)=\infty$, which contradicts (i).

Assume now that $c_*(A)=\infty$; we need to construct $\omega\in\mathcal E^+_A$ with $\omega(X)=\infty$.

Choose an increasing sequence of compact sets $K_j\subset A$, $j\in\mathbb N$, such that
\begin{equation}\label{a}\lim_{j\to\infty}\,c(K_j)=c_*(A)=\infty.\end{equation}
Since $c(K_j)<\infty$ by the energy principle, an application of Lemma~\ref{str} gives
\[c_*(A\setminus K_j)=\infty\text{ \ for all $j$}.\]
Therefore, there exists a compact set $L_j\subset A\setminus K_j$ with $c(L_j)>j^4$, or equivalently
\[w(L_j)<j^{-4}.\]
We can certainly assume that $L_j\subset K_{j+1}$, for if not, we replace $K_{j+1}$ by $K_{j+1}\cup L_j$, which does not affect limit relation (\ref{a}). Define
\[\omega(f):=\sum_{j\in\mathbb N}\,\lambda_{L_j}(f)\text{ \ for all $f\in C_0(X)$},\]
where $\lambda_{L_j}$ is the (unique) solution to extremal problem (\ref{W}) for the set $L_j$. Then $\omega$ represents a positive Radon measure on $X$, because any compact subset of $X$ has points in common with at most finitely many $L_j$. Furthermore, the measure $\omega$ is concentrated on the union of $L_j$, $j\in\mathbb N$, and hence on the (larger) set $A$. Besides,
\[\omega(X)=\sum_{j\in\mathbb N}\,\lambda_{L_j}(X)=\sum_{j\in\mathbb N}\,1=+\infty,\]
the sets $L_j$ being mutually disjoint. It thus remains to verify that
\begin{equation}\label{enf}
 \kappa(\omega,\omega)<\infty.
\end{equation}

To this end, observe that $\omega$ is the vague limit of the partial sums
$\lambda_N:=\sum_{j=1}^N\,\lambda_{L_j}$
(as $N\to\infty$), because for every $f\in C_0(X)$, $\omega(f)=\lambda_N(f)$ for all $N$ sufficiently large. Noting that, by the triangle inequality,
\[\|\lambda_N\|\leqslant\sum_{j=1}^N\,\|\lambda_{L_j}\|<\sum_{j=1}^N\,j^{-2}<\sum_{j\in\mathbb N}\,j^{-2}<\infty\text{ \ for all $N$},\]
we therefore conclude by the principle of descent that
\[\kappa(\omega,\omega)\leqslant\liminf_{N\to\infty}\,\kappa(\lambda_N,\lambda_N)<\infty,\]
whence (\ref{enf}). This establishes (ii)$\Rightarrow$(i), thereby completing the whole proof.

\subsection{An alternative proof of Theorem~\ref{th-ap}}\label{p} If $c_*(A)=0$, (\ref{G0}) is trivial, for then $\mathcal E^+_A=\{0\}$ by Lemma~\ref{l-negl1}. If now $c_*(A)=\infty$, there exists $\nu_0\in\mathcal E^+_A$ with $\nu_0(X)=\infty$ (Theorem~\ref{capf}), and (\ref{G0}) holds in the sense that its both sides are equal to $+\infty$.

It is left to verify (\ref{G0}) in the remaining case $0<c_*(A)<\infty$; then, again by Theorem~\ref{capf}, all $\nu\in\mathcal E^+_A$ must be {\it bounded}. Excluding the measure $\nu=0$, write $\nu=q\mu$, where $\mu\in\breve{\mathcal E}^+_A$ and $q:=\nu(X)\in(0,\infty)$. Then $G(\nu)$ takes the form $2q-q^2\|\mu\|^2$,
which attains its maximum in $q\in(0,\infty)$ at $q_0:=1/\|\mu\|^2\in(0,\infty)$ (the kernel being strictly positive definite), and this maximum equals $q_0=1/\|\mu\|^2\in(0,\infty)$.
Taking now the supremum over $\mu\in\breve{\mathcal E}^+_A$ we obtain
\begin{equation}\label{minsup}\sup_{\nu\in\mathcal E^+_A}\,G(\nu)=\Bigl[\inf_{\mu\in\breve{\mathcal E}^+_A}\,\|\mu\|^2\Bigr]^{-1}=c_*(A),\end{equation}
and the proof is complete.

\section{Basic facts on inner capacities and inner capacitary measures}\label{sec-2}

In the rest of the paper we shall assume a kernel $\kappa$ to be {\it perfect\/} (Sect.~\ref{sec-1}).

\subsection{Inner capacitary measures} Suppose for a moment that a set $A=K\subset X$ is compact, and  $w(K)<\infty$, cf.\ (\ref{W}). Then there exists the unique capacitary distribution $\lambda_K$ of unit mass on $K$, minimizing the energy over the class $\breve{\mathcal E}^+_K$:
\[\lambda_K\in\breve{\mathcal E}^+_K\text{ \ and \ }\|\lambda_K\|^2=w(K).\]
In fact, the existence of this $\lambda_K$ is due to the vague compactness of the class $\breve{\mathcal E}^+_K$, whereas the uniqueness follows easily by means of standard arguments based on the convexity of the class $\breve{\mathcal E}^+_K$ and a pre-Hilbert structure on the space $\mathcal E$ (cf.\ footnote~\ref{unique}).
Noting that then necessarily $w(K)>0$ by the energy principle, one can also consider the (normalized) capacitary measure $\gamma_K$ on the (compact) set $K$, introduced by
\begin{equation}\label{G}\gamma_K:=\lambda_K/w(K)=c(K)\lambda_K\in\mathcal E^+_K\end{equation}
and having properties (\ref{eq-pr-c0})--(\ref{eq-pr-c3}).

However, if $A\subset X$ is noncompact, $\breve{\mathcal E}^+_A$ is no longer vaguely compact, and the infimum in (\ref{W})
is in general not attained among the admissible measures~--- even for a perfect kernel.
To overcome this inconvenience, Fuglede has formulated the (dual) minimum energy problem over the class $\Gamma_A\subset\mathcal E$ (see below), which leads to the same concept of inner capacity, but it is already solvable (see \cite[Theorem~4.1]{F1}).

Given $A\subset X$, let $\Gamma_A$ denote the class of all (signed) $\nu\in\mathcal E$ having the property $\kappa\nu\geqslant1$ n.e.\ on $A$, and $\Gamma_A^+$ its subclass consisting of all {\it positive\/} measures:
\begin{equation}\label{def-in}\Gamma_A^+:=\bigl\{\nu\in\mathcal E^+: \ \kappa\nu\geqslant1\text{ \ n.e.\ on $A$}\bigr\}.\end{equation}
Due to the energy principle accepted in the present section, \cite[Theorem~4.1]{F1} can be specified as follows.

\begin{theorem}\label{prop.1.2'}For any\/ $A\subset X$,
\begin{equation*}\label{I}\inf_{\nu\in\Gamma_A^+}\,\|\nu\|^2=\inf_{\nu\in\Gamma_A}\,\|\nu\|^2=c_*(A).\end{equation*}
If now\/ $c_*(A)<\infty$, then each of these two infima is an actual minimum with the same unique minimizing measure\/ $\gamma_A\in\Gamma_A^+\subset\Gamma_A$, called the inner capacitary measure of\/ $A$ and having the properties\/\footnote{It follows that $c_*(A)<\infty\iff\Gamma_A^+\ne\varnothing\iff\Gamma_A\ne\varnothing$.\label{f5}}
\begin{align}
\label{pr1}&\gamma_A(X)=\|\gamma_A\|^2=c_*(A),\\
\label{pr0}&\kappa\gamma_A\geqslant1\text{ \ n.e.\ on\/ $A$},\\
\label{pr2}&\kappa\gamma_A\leqslant1\text{ \ on\/ $S(\gamma_A)$},\\
\label{pr3}&\kappa\gamma_A=1\text{ \ $\gamma_A$-a.e.\ on\/ $X$.}\end{align}
If moreover Frostman's maximum principle is fulfilled, then also\/\footnote{Then $\gamma_A$ is also said to be the inner {\it equilibrium\/} measure of $A$ (cf.\ footnote~\ref{fr}).}
\begin{equation}\label{F}\kappa\gamma_A=1\text{ \ n.e.\ on\/ $A$.}\end{equation}
\end{theorem}

\begin{proof}By \cite[Lemma~3.2.2]{F1} with $t=1$, $\|\nu\|^2\geqslant c_*(A)$ for all $\nu\in\Gamma_A$, and hence
\begin{equation}\label{ineq}\inf_{\nu\in\Gamma_A^+}\,\|\nu\|^2\geqslant\inf_{\nu\in\Gamma_A}\,\|\nu\|^2\geqslant c_*(A).\end{equation}
It is thus enough to consider the case where $c_*(A)<\infty$.

The uniqueness of the solution to the minimum energy problem over $\Gamma_A$ $(\Gamma_A^+)$ follows from the convexity of $\Gamma_A$ $(\Gamma_A^+)$, which in turn is obvious from Lemma~\ref{str}.

The (unique) measure $\gamma\in\Gamma_A$ of minimal energy in the class $\Gamma_A$ does exist, and moreover it satisfies (\ref{pr1})--(\ref{pr3}) (see \cite[Theorem~4.1]{F1}). It thus remains to prove that the same $\gamma$ solves the minimum energy problem over $\Gamma_A^+$. In view of (\ref{ineq}), this will follow once we show that this $\gamma$ is positive.
Assume $c_*(A)>0$, for if not, $\gamma=\gamma_A=0$.

Analyzing the proof of Theorem~4.1 in \cite{F1}, we deduce the following.

\begin{lemma}\label{O}The above measure\/ $\gamma$ is the strong limit of\/ $(\gamma_{K_j})$, where\/ $(K_j)$ is an increasing sequence of compact sets\/ $K_j\subset A$ with\/ $c(K_j)>0$ and such that\/\footnote{Such a sequence $(K_j)$ exists by virtue of (\ref{153}).}
\begin{equation}\label{limm}\lim_{j\to\infty}\,c(K_j)=c_*(A),\end{equation}
while\/ $\gamma_{K_j}=\lambda_{K_j}/w(K_j)\in\mathcal E^+_{K_j}$ is the\/ {\rm(}unique\/{\rm)} capacitary measure on\/ $K_j$.
\end{lemma}

Since the kernel is perfect, $\gamma_{K_j}\to\gamma$ also vaguely. Hence $\gamma$ is indeed positive, the cone $\mathfrak M^+$ being vaguely closed \cite[Section~III.1.9, Proposition~14]{B2}.
\end{proof}

\begin{remark}\label{rem4}In general, $\gamma_A$ is {\it not\/} concentrated on the set $A$ itself, but on its closure in $X$.
For instance, if $A:=B_r$ is an open ball $\{|x|<r\}$, $r>0$, in $\mathbb R^n$, $n\geqslant3$, and $\kappa(x,y)$ is the Newtonian kernel $|x-y|^{2-n}$, then $\gamma_{B_r}$ is the positive measure of total mass $r^{n-2}$, uniformly distributed over the sphere $\{|x|=r\}$. Thus $S(\gamma_{B_r})\cap B_r=\varnothing$.
\end{remark}

We shall now show that the problem of the determination of the inner capacitary measure for arbitrary $A$ can always be reduced to that for a (Borel) $K_\sigma$-set $A'\subset A$.\footnote{A set $A\subset X$ is said to be of class $K_\sigma$ (in short, a $K_\sigma$-set) if it is representable as a countable union of compact subsets.} This observation will be particularly useful in the proof of Theorem~\ref{th-cont2} (Sect.~\ref{proof-th-cont2}), dealing with the convergence of inner capacitary measures and their potentials.

\begin{theorem}\label{th-count}For arbitrary\/ $A\subset X$, there exists a\/ $K_\sigma$-set\/ $A'\subset A$ such that, for every\/ $H$ with the property\/ $A'\subset H\subset A$,
\begin{equation}\label{eq2}c_*(A')=c_*(H)=c_*(A).\end{equation}
If now\/ $c_*(A)<\infty$, then also
\begin{equation}\label{eq3}\gamma_{A'}=\gamma_H=\gamma_A.\end{equation}
\end{theorem}

\begin{proof} Given $A\subset X$, choose an increasing sequence $(K_j)\subset A$ of compact sets such that $c(K_j)\to c_*(A)$,
and let $A'$ denote the union of all those $K_j$.
By (\ref{153}) applied to $A'$ (or alternatively by \cite[Lemma~2.3.3]{F1}, the sets $K_j$ being universally measurable),
\[c_*(A')=\lim_{j\to\infty}\,c(K_j),\]
and hence $c_*(A')=c_*(A)$. For any $H$ such that $A'\subset H\subset A$, (\ref{eq2}) now follows
from the monotonicity of the inner capacity:
\[c_*(A')\leqslant c_*(H)\leqslant c_*(A)=c_*(A').\]

Assume moreover that $c_*(A)<\infty$. Noting that $\Gamma_A^+\subset\Gamma_H^+\subset\Gamma_{A'}^+$, we derive (\ref{eq3}) from (\ref{eq2}) by use of the uniqueness of the inner capacitary measure.
\end{proof}

\subsection{Relations between the minimum energy problems over $\Gamma^+_A$ and $\breve{\mathcal E}^+_A$}\label{sec-rel}
Despite the fact that the minimum energy problem over the class $\breve{\mathcal E}^+_A$ is in general unsolvable, it is often useful to consider it in parallel with the (dual) minimum energy problem over the class $\Gamma^+_A$.
To analyze relations between these two problems, assume $w(A)<\infty$, for then (and only then) $\breve{\mathcal E}^+_A$ is nonempty, and problem (\ref{W}) makes sense.
To avoid trivialities, assume also that $w(A)>0$, for if not, problem (\ref{W}) admits no solution by the energy principle. Thus, in the rest of this section,
\begin{equation}\label{ASS}0<c_*(A)<\infty,\end{equation}
and hence there exists the (unique) inner capacitary measure $\gamma_A\ne0$ (Theorem~\ref{prop.1.2'}).

\begin{definition}\label{def-min}
  A net $(\mu_s)_{s\in S}\subset\breve{\mathcal E}^+_A$ is said to be {\it minimizing\/} in problem (\ref{W}) if
\begin{equation}\label{8}\lim_{s\in S}\,\|\mu_s\|^2=w(A).\end{equation}
\end{definition}

As $\breve{\mathcal E}^+_A$ is nonempty, so is the set $\mathbb M(A)$ consisting of all minimizing nets $(\mu_s)_{s\in S}$.

\begin{lemma}\label{l-extr}For every\/ $(\mu_s)_{s\in S}\in\mathbb M(A)$,
\begin{equation}\label{limm'}\mu_s\to\xi_A\text{ \ strongly and vaguely},\end{equation}
where
\begin{equation}\label{extrcap}\xi_A:=\gamma_A/c_*(A)\in\mathcal E^+.
\end{equation}
\end{lemma}

\begin{proof} We shall first show that for any $(\mu_s)_{s\in S}$ and $(\nu_t)_{t\in T}$
from $\mathbb M(A)$,
\begin{equation}\label{ST}
\lim_{(s,t)\in S\times T}\,\|\mu_s-\nu_t\|=0,
\end{equation}
$S\times T$ being the directed product of the directed sets $S$ and
$T$ (see e.g.\ \cite[p.~68]{K}).
In fact, due to the convexity of the class $\breve{\mathcal E}^+_A$, for any $(s,t)\in S\times T$ we have
\[\|\mu_s+\nu_t\|^2\geqslant4w(A).\]
Applying the parallelogram identity in $\mathcal E$ therefore gives
\[
0\leqslant\|\mu_s-\nu_t\|^2\leqslant-4w(A)+2\|\mu_s\|^2+2\|\nu_t\|^2,\]
which yields (\ref{ST}) when combined with (\ref{8}).

Taking the two nets in (\ref{ST}) to be equal, we see that every $(\nu_t)_{t\in T}\in\mathbb M(A)$ is strong Cauchy. But $(\nu_t)_{t\in T}$ is vaguely bounded, for $\nu_t(X)=1$ for all $t\in T$, and hence vaguely relatively compact \cite[Section~III.1.9, Proposition~15]{B2}. If $\nu_0$ denotes a vague limit point of $(\nu_t)_{t\in T}$, then  $\nu_t\to\nu_0$ strongly, the kernel $\kappa$ being perfect. The same unique $\nu_0$ also serves as the strong limit of any other net $(\mu_s)_{s\in S}\in\mathbb M(A)$, which is obvious from (\ref{ST}). Since the strong convergence implies the vague convergence,
\[\mu_s\to\nu_0\text{ \ strongly and vaguely}.\]

To prove that this $\nu_0$ actually equals $\xi_A$, $\xi_A$ being introduced by (\ref{extrcap}), it is thus enough to construct a minimizing net converging to $\xi_A$ strongly. Let a sequence $(K_j)$ of compact subsets of $A$ be as in Lemma~\ref{O}; it is clear from (\ref{limm}) that
\[(\lambda_{K_j})\in\mathbb M(A),\]
$\lambda_{K_j}\in\breve{\mathcal E}^+_{K_j}\subset\breve{\mathcal E}^+_A$ solving problem (\ref{W}) for the set $K_j$. But according to Lemma~\ref{O},
$c(K_j)\lambda_{K_j}\to\gamma_A$ strongly, which combined with (\ref{limm}) and (\ref{extrcap}) shows that, indeed,
\[\lambda_{K_j}\to\xi_A\text{ \ strongly},\]
thereby completing the proof of the lemma.
\end{proof}

\begin{remark}\label{extr-rem}
In view of (\ref{limm'}), it is natural to call the measure $\xi_A\in\mathcal E^+$ {\it extremal\/} in problem (\ref{W}). It is obvious from (\ref{pr1}) and (\ref{extrcap}) that
\begin{equation}\label{and}\xi_A(X)=1,\quad\|\xi_A\|^2=w(A).\end{equation}
Nevertheless, the extremal measure $\xi_A$ may not be the solution to problem (\ref{W}), being in general not concentrated on $A$ (cf.\  Remark~\ref{rem4}).
 \end{remark}

\begin{theorem}\label{cor1}Under assumption\/ {\rm(\ref{ASS})}, the three assertions are equivalent:
\begin{itemize}\item[{\rm(i)}] The extremal measure\/ $\xi_A$ is concentrated on\/ $A$.
\item[{\rm(ii)}]Problem\/ {\rm(\ref{W})} has the\/ {\rm(}unique\/{\rm)} minimizing measure\/ $\lambda_A\in\breve{\mathcal E}^+_A$:
\[\|\lambda_A\|^2=\min_{\nu\in\breve{\mathcal E}^+_A}\,\|\nu\|^2,\]
called the capacitary distribution of unit mass on\/ $A$.
\item[{\rm(iii)}] The inner capacitary measure\/ $\gamma_A$ serves simultaneously as the\/ {\rm(}unique\/{\rm)} solution to the minimum energy problem over the class\/ $\bigl\{\nu\in\mathcal E^+_A: \ \nu(X)=c_*(A)\bigr\}$.
\end{itemize}
If any of these\/ {\rm(i)}--{\rm(iii)} holds true, then actually
\begin{equation}\label{3}\lambda_A=\xi_A=\gamma_A/c_*(A).
\end{equation}
\end{theorem}

\begin{proof} The fact that (i) implies (ii) (with $\lambda_A:=\xi_A$)
follows directly from (\ref{and}).

Let now (ii) hold. Since the trivial sequence $(\lambda_A)$ is minimizing:
$(\lambda_A)\in\mathbb M(A)$,
we conclude from Lemma~\ref{l-extr} that then necessarily $\lambda_A=\xi_A$ (the strong topology on $\mathcal E$ being Hausdorff), whence (i). Also, combining $\lambda_A=\xi_A$ with (\ref{extrcap}) gives (\ref{3}).

Finally, the equivalence of (ii) and (iii) is obvious
for reasons of homogeneity.
\end{proof}

\begin{corollary}\label{cor1qu}Assume\/ {\rm(\ref{ASS})} is fulfilled. For\/ {\rm(i)}--{\rm(iii)} in Theorem\/~{\rm\ref{cor1}} to hold, it is sufficient that\/ $\mathcal E^+_A$ be strongly closed\/ {\rm(}or in particular that\/ $A$ be quasiclosed\/{\rm)}.\end{corollary}

\begin{proof} If $\mathcal E^+_A$ is strongly closed (or in particular if $A$ is quasiclosed, cf.\ Lemma~\ref{l-quasi}), then the extremal measure, being the strong limit of any minimizing net, must be concentrated on $A$. Thus (i) holds true, and hence so do both (ii) and (iii).
\end{proof}

\section{Further alternative characterizations of $\gamma_A$}\label{sec-ch}

As in Sect.~\ref{sec-2}, in the present section a kernel $\kappa$ is assumed to be {\it perfect}.
We first establish necessary and sufficient conditions for (extremal) problems (\ref{G0}) and (\ref{G00}) to be uniquely solvable (Theorem~\ref{th-ch}), which will further be utilized to characterize the inner capacitary measure $\gamma_A$ for
 $A$ with the strongly closed $\mathcal E^+_A$ (Theorem~\ref{cor3}).

\begin{theorem}\label{th-ch}For any\/ $A\subset X$ with\/ $0<c_*(A)<\infty$,\footnote{If $c_*(A)=0$, then $\mathcal E^+_A=\widehat{\mathcal E}^+_A=\{0\}$ (Lemma~\ref{l-negl1}), and problems (\ref{G0}) and (\ref{G00}) are obviously solvable with the same unique solution $\theta_A=\sigma_A=0$ \ $({}=\gamma_A)$.}
each of assertions\/ {\rm(i)}--{\rm(iii)} in Theorem\/~{\rm\ref{cor1}} is equivalent to either of the following\/ {\rm(iv)} and\/ {\rm(v)}:
\begin{itemize}
  \item[{\rm(iv)}] There exists a solution\/ $\theta_A$ to problem\/~{\rm(\ref{G0})}, that is, $\theta_A\in\mathcal E^+_A$ and
  \[G(\theta_A)=\max_{\nu\in\mathcal E^+_A}\,G(\nu)\quad\bigl({}=c_*(A)\bigr).\]
  \item[{\rm(v)}] There exists a solution\/ $\sigma_A$ to problem\/~{\rm(\ref{G00})}, that is, $\sigma_A\in\widehat{\mathcal E}^+_A$ and
  \[\sigma_A(X)=\max_{\nu\in\widehat{\mathcal E}^+_A}\,\nu(X)\quad\bigl({}=c_*(A)\bigr).\]
  \end{itemize}
  If any of these\/ {\rm(i)}--{\rm(v)} is fulfilled, then\/ $\theta_A$ and\/ $\sigma_A$ are unique, and moreover
  \[\theta_A=\sigma_A=\gamma_A=c_*(A)\lambda_A.\]
  \end{theorem}

\begin{proof}
It follows from the proof given in Sect.~\ref{p} (see in particular formula (\ref{minsup})) that the solution $\theta_A$ to problem (\ref{G0}) exists if and only if so does
the solution $\lambda_A$ to problem (\ref{W}), and in the affirmative case they are related to one another as follows:
\[\theta_A=c_*(A)\lambda_A=\gamma_A,\]
the latter equality being valid by (\ref{3}). This proves the equivalence of (ii) and (iv), as well as the uniqueness of $\theta_A$, the solution $\lambda_A$ to problem (\ref{W}) being unique.

Furthermore, (iv) implies (v) with $\sigma_A:=\gamma_A=\theta_A$. Indeed,
$\gamma_A\in\widehat{\mathcal E}^+_A$ by (\ref{pr3}), which together with (\ref{G00}) and (\ref{pr1}) gives
\[\gamma_A(X)\leqslant\sup_{\nu\in\widehat{\mathcal E}^+_A}\,\nu(X)=c_*(A)=\gamma_A(X).\]

Finally, if (v) is fulfilled for some $\sigma_A\in\widehat{\mathcal E}^+_A$, then, by (\ref{estt}) with $\nu:=\sigma_A$,
\[c_*(A)=\sigma_A(X)\leqslant G(\sigma_A)\leqslant\sup_{\nu\in\mathcal E^+_A}\,G(\nu)=c_*(A),\]
and hence the same $\sigma_A$ must be the (unique) solution to problem (\ref{G0}).
\end{proof}

\begin{theorem}\label{cor3} Given a set\/ $A\subset X$ with\/ $c_*(A)<\infty$, assume that the class\/ $\mathcal E^+_A$ is strongly closed\/ {\rm(}or in particular that\/ $A$ is quasiclosed\/{\rm)}. Then
the inner capacitary measure\/ $\gamma_A$ can alternatively be characterized as the unique solution to either of extremal problems\/ {\rm(\ref{G0})} or\/ {\rm(\ref{G00})}; that is, $\gamma_A\in\widehat{\mathcal E}^+_A\subset\mathcal E^+_A$ and
  \begin{align}
    &\max_{\nu\in\mathcal E^+_A}\,G(\nu)=G(\gamma_A)\quad\bigl({}=c_*(A)\bigr),\label{cl1}\\
    &\max_{\nu\in\widehat{\mathcal E}^+_A}\,\nu(X)=\gamma_A(X)\quad\bigl({}=c_*(A)\bigr).\label{cl2}
  \end{align}
Furthermore, $\gamma_A$ is uniquely determined within\/ $\mathcal E^+_A$ by each of\/ {\rm(a)} or\/ {\rm(b)}, where
\begin{itemize}
\item[{\rm(a)}] $\gamma_A(X)\geqslant\|\gamma_A\|^2\geqslant c_*(A)$.
\item[{\rm(b)}] $\kappa\gamma_A\geqslant1$ n.e.\ on $A$, and\/ $\kappa\gamma_A\leqslant1$ $\gamma_A$-a.e.
\end{itemize}
If moreover Frostman's maximum principle is fulfilled, then\/ $\gamma_A$ is the only measure in\/ $\mathcal E^+_A$ having the property
\begin{itemize}
\item[{\rm(c)}] $\kappa\gamma_A=1$ n.e.\ on\/ $A$.
\end{itemize}
\end{theorem}

\begin{proof}Let the hypotheses of the theorem be fulfilled. Also assume that $c_*(A)>0$, for if not, $\mathcal E^+_A=\widehat{\mathcal E}^+_A=\{0\}$  (Lemma~\ref{l-negl1}), and the theorem holds trivially with $\gamma_A=0$.
Taking Corollary~\ref{cor1qu} into account, we then conclude from Theorem~\ref{th-ch} that the inner capacitary measure $\gamma_A$ can alternatively be characterized as the unique maximizing measure in either of problems (\ref{G0}) and (\ref{G00}), which establishes (\ref{cl1}) and (\ref{cl2}).

Besides, it is obvious from (\ref{pr1}), (\ref{pr0}), and (\ref{pr3}) that $\gamma_A$ has properties (a) and (b). If now $\mu$ is another measure in $\mathcal E^+_A$ satisfying (a), then, in view of (\ref{G0}),
\[c_*(A)=\sup_{\nu\in\mathcal E^+_A}\,G(\nu)\geqslant G(\mu)=2\mu(X)-\|\mu\|^2\geqslant\|\mu\|^2\geqslant c_*(A).\]
Being therefore the maximizing measure in problem (\ref{G0}), $\mu$ necessarily equals $\gamma_A$.

Assume now that some $\chi\in\mathcal E^+_A$ meets (b). By the former relation in (b), we have $\chi\in\Gamma^+_A$, hence $\|\chi\|^2\geqslant c_*(A)$, whereas the latter relation in (b) yields
\[\|\chi\|^2=\int\kappa\chi\,d\chi\leqslant\chi(X).\] Thus (b) is actually reduced to (a), whence  $\chi=\gamma_A$ (cf.\ the preceding paragraph).

Assume finally that Frostman's maximum principle holds; then $\gamma_A$ satisfies (c) (Theorem~\ref{prop.1.2'}).
Given $\tau\in\mathcal E^+_A$ meeting (c), we need to prove that then necessarily $\tau=\gamma_A$, which would follow immediately if $\tau$ were shown to fulfill (b), or equivalently
\begin{equation}\label{tauu}\kappa\tau\leqslant1\text{ \ $\tau$-a.e.}\end{equation}
We may verify this only for positive $\kappa$, because the remaining case of compact $X$ can be reduced to the former case in the manner described in footnote~\ref{fIII}.

For any compact $K\subset A$, then $\kappa\tau|_K\leqslant\kappa\tau$ on $X$; therefore, $\kappa\tau|_K\leqslant1$ holds n.e.\ on $K$, hence everywhere on $S(\tau|_K)\subset K$, the potential of a positive measure being l.s.c.\ on $X$, and an application of Frostman's maximum principle gives
\[\kappa\tau|_K\leqslant1\text{ \ on $X$}.\]
Since $\tau|_K\to\tau|_A=\tau$ vaguely as $K\uparrow A$ (cf.\ Lemma~\ref{l-lower} with $\psi:=f\in C^+_0(X)$),\footnote{As usual, $C_0^+(X)$ denotes the subclass of $C_0(X)$ consisting of all $f\geqslant0$.}
\[\kappa\tau\leqslant\liminf_{K\uparrow A}\,\kappa\tau|_K\leqslant1\text{ \ on $X$}\]
(the principle of descent), whence (\ref{tauu}).\end{proof}

\begin{remark}\label{closedness}
  The requirement of the strong closedness of the class $\mathcal E^+_A$ is essential for the validity of Theorem~\ref{cor3}. Indeed, if $\mathcal E^+_A$ is strongly unclosed, it may happen that $\gamma_A\not\in\mathcal E^+_A$ (see  Remark~\ref{rem4}), and then both (\ref{cl1}) and (\ref{cl2}) do fail to hold. Furthermore, even if Frostman's maximum principle is fulfilled, there may exist no $\nu\in\mathcal E^+_A$ with $\kappa\nu=1$ n.e.\ on $A$ (see footnote~\ref{Foot}). In view of these facts, Theorem~\ref{G0'ar}, generalizing Theorem~\ref{cor3} to {\it arbitrary\/} $A$, is of particular interest to the present study.
\end{remark}

\section{Convergence of inner capacitary measures and their potentials}\label{sec-conv}

As in Sects.~\ref{sec-2} and \ref{sec-ch}, in this section a kernel $\kappa$ is assumed to be {\it perfect}.

\begin{theorem}\label{cor2}For arbitrary\/ $A\subset X$ with\/ $c_*(A)<\infty$,
\begin{equation}\label{eq4}\gamma_K\to\gamma_A\text{ \ strongly and vaguely in\/ $\mathcal E^+$ as\/ $K\uparrow A$.}\end{equation}
If moreover the first and the second maximum principles both hold,
then also
\begin{equation*}\label{eq5}\kappa\gamma_K\uparrow\kappa\gamma_A\text{ \ pointwise on\/ $X$ as\/ $K\uparrow A$,}\end{equation*}
and hence
\begin{equation}\label{eq6}\kappa\gamma_A=\sup_{K\in\mathfrak C_A}\,\kappa\gamma_K\text{ \ on\/ $X$.}\end{equation}
\end{theorem}

\begin{corollary}\label{prop.1.11} For any\/ $A,H\subset X$ such that\/ $H\subset A$ and\/ $c_*(A)<\infty$,
\begin{equation}\label{eq.1.11}\|\gamma_A-\gamma_H\|^2\leqslant\|\gamma_A\|^2-\|\gamma_H\|^2.\end{equation}
If moreover Frostman's maximum principle is fulfilled, equality prevails here:
\begin{equation}\label{eq.1.111}\|\gamma_A-\gamma_H\|^2=\|\gamma_A\|^2-\|\gamma_H\|^2.\end{equation}
\end{corollary}

\begin{proof}Since $\gamma_A\in\Gamma_H^+$, while $\gamma_H$ minimizes $\|\nu\|^2$ over $\nu\in\Gamma_H^+$, we derive (\ref{eq.1.11}) from \cite[Lemma~4.1.1]{F1}.
Let now Frostman's maximum principle be satisfied. Noting that then $\kappa\gamma_A=1$ holds n.e.\ on each $K\in\mathfrak C_H$, hence $\gamma_K$-a.e.\ (Lemma~\ref{l-negl}), whereas
\begin{equation}\label{eq.h}
\gamma_K\to\gamma_H\text{ \ strongly in $\mathcal E^+$ as $K\uparrow H$},\end{equation}
we obtain
\begin{align*}\kappa(\gamma_A,\gamma_H)=\lim_{K\uparrow H}\,\kappa(\gamma_A,\gamma_K)=\lim_{K\uparrow H}\,\int\kappa\gamma_A\,d\gamma_K=\lim_{K\uparrow H}\,\gamma_K(X)=\lim_{K\uparrow H}\,\|\gamma_K\|^2=\|\gamma_H\|^2,\end{align*}
whence (\ref{eq.1.111}). (The forth equality here is valid by (\ref{pr1}) with $A:=K$, while the first follows from (\ref{eq.h}) with the aid of the Cauchy--Schwarz inequality.)\end{proof}

\begin{corollary}[{\rm Monotonicity property}]\label{eq-mon} If\/ $c_*(A)<\infty$, and if the first and the second maximum principles both hold, then for any\/ $H\subset A$,
\begin{equation}\label{eq7}\kappa\gamma_H\leqslant\kappa\gamma_A\text{ \ on\/ $X$}.\end{equation}
\end{corollary}

\begin{proof}Since obviously $\mathfrak C_H\subset\mathfrak C_A$, (\ref{eq7}) is implied by (\ref{eq6}).\end{proof}

\begin{theorem}\label{th-cont2}Given arbitrary\/ $A\subset X$, assume there exists an increasing sequence\/ $(U_j)$ of universally measurable sets\/ $U_j\subset X$ such that\/\footnote{Such $(U_j)$ necessarily exists if the space $X$ is second-countable or, more generally, $\sigma$-com\-pact.}
\[A=\bigcup_{j\in\mathbb N}\,A_j,\text{ \ where\/ $A_j:=A\cap U_j$.}\]
Then
\begin{equation}\label{eq-cont}c_*(A)=\lim_{j\to\infty}\,c_*(A_j).\end{equation}
If now\/ $c_*(A)<\infty$, then also
\begin{align}&\gamma_{A_j}\to\gamma_A\text{ \ strongly and vaguely},\label{eq-cont1}\\
&\kappa\gamma_A=\liminf_{j\to\infty}\,\kappa\gamma_{A_j}\text{ \ n.e.\ on\/ $X$.}\label{eq-cont2}
\end{align}
If moreover the first and the second maximum principles both hold,
then\/ {\rm(\ref{eq-cont2})} can be refined as follows:
\begin{equation}\label{eq-cont3}\kappa\gamma_{A_j}\uparrow\kappa\gamma_A\text{ \ on\/ $X$.}\end{equation}
\end{theorem}

\begin{remark}Limit relations (\ref{eq-cont}) and (\ref{eq-cont1}) generalize Fuglede's results \cite{F1} (Lemma~2.3.3 and Theorem~4.2), given in the case where $A$ is the union of an increasing sequence of universally measurable subsets (and hence $A$ is likewise universally measurable).\footnote{See also \cite[Section~II.2.9, Remark]{L} pertaining to the Riesz kernels on $\mathbb R^n$.} This generalization (see Sect.~\ref{proof-th-cont2} for a proof) is substantially based on Theorem~\ref{th-count}, which makes it possible to reduce the analysis to that for Borel sets.
\end{remark}

\begin{theorem}\label{th-cont-bor2}Let\/ $A$ be the intersection of a decreasing net\/ {\rm(}resp.\ a decreasing sequence\/{\rm)} $(A_t)_{t\in T}$ of closed\/ {\rm(}resp.\ quasiclosed\/{\rm)} sets\/ $A_t\subset X$ such that
\begin{equation}\label{111}c_*(A_{t_0})<\infty\text{ \ for some\/ $t_0\in T$}.\end{equation}
Then\/\footnote{If $A$ is the intersection of a decreasing net of closed $A_t$, (\ref{eq-cl}) was declared in \cite[Lemma~4.2.1]{F1}; however, its proof was incomplete, being based on \cite[Section~2.5, Eq.~(1)]{F1} (cf.\ Remark~\ref{Fuu}).\label{f-decr}}
\begin{equation}\gamma_{A_t}\to\gamma_A\text{ \ strongly and vaguely},\label{eq-cl}\end{equation}
and hence
\begin{equation}\lim_{t\in T}\,c_*(A_t)=c_*(A).\label{eq-cl-cap}\end{equation}
If moreover the first and the second maximum principles both hold,
then also
\begin{equation}\label{eq-cl-pot}\kappa\gamma_{A_t}\downarrow\kappa\gamma_A\text{ \ pointwise n.e.\ on\/ $X$}.\end{equation}
\end{theorem}

\begin{remark}Assumption (\ref{111}) is essential for the validity of Theorem~\ref{th-cont-bor2}, and this is the case even for the Newtonian kernel $|x-y|^{2-n}$ on $\mathbb R^n$, $n\geqslant3$. For instance, consider $A_j:=K\cup\{|x|\geqslant j\}$, $j\in\mathbb N$, a set $K\subset\mathbb R^n$ being compact. Then the sequence $(A_j)$ of the closed sets $A_j$ decreases, and its intersection equals $K$. However, $c(K)<\infty$, whereas $c(A_j)=\infty$ for all $j$, and so (\ref{eq-cl-cap}) indeed fails to hold.\end{remark}

\subsection{Proof of Theorem~\ref{cor2}} Lemma~\ref{l-extr} is crucial to the present proof.

Assume $c_*(A)>0$, for if not, the theorem holds trivially with $\gamma_K=\gamma_A=0$. Then for all $K\in\mathfrak C_A$ sufficiently large ($K\geqslant K_0$), we have $c(K)>0$, and hence there exists the (unique) $\lambda_K\in\breve{\mathcal E}^+_K\subset\breve{\mathcal E}^+_A$ with
\[\|\lambda_K\|^2=w(K)=1/c(K).\]
It is clear from (\ref{153}) that these $\lambda_K$ form a minimizing net,
which according to Lemma~\ref{l-extr} must converge to the extremal measure $\xi_A$ strongly and vaguely. Combined with (\ref{153}), (\ref{G}), and (\ref{extrcap}), this
establishes (\ref{eq4}).

Assume now additionally that the first and the second maximum principles are both fulfilled. For any $K,K'\in\mathfrak C_A$ such that $K_0\leqslant K\leqslant K'$, then
\[\kappa\gamma_K=\kappa\gamma_{K'}=\kappa\gamma_A=1\text{ \ n.e.\ on $K$},\]
which is derived from (\ref{F}) by use of the countable subadditivity of inner capacity on universally measurable sets. Since each of these equalities holds $\gamma_K$-a.e.\ on $X$ (Lemma~\ref{l-negl}), the second maximum principle gives
\[\kappa\gamma_K\leqslant\kappa\gamma_{K'}\leqslant\kappa\gamma_A\text{ \ on $X$}.\]
Thus $\lim_{K\uparrow A}\,\kappa\gamma_K$ exists everywhere on $X$, and moreover it does not exceed $\kappa\gamma_A$. To complete the proof, it remains to show that
\[\kappa\gamma_A\leqslant\lim_{K\uparrow A}\,\kappa\gamma_K\text{ \ on $X$},\]
which however follows immediately from the vague convergence of the net $(\gamma_K)_{K\in\mathfrak C_A}$ to $\gamma_A$ by the principle of descent.

\subsection{Proof of Theorem~\ref{th-cont2}}\label{proof-th-cont2}
Applying Theorem~\ref{th-count} we see that there exist $K_\sigma$-sets $A_j'\subset A_j$, $j\in\mathbb N$, and $A'\subset A$ such that $A_j'\subset A_{j+1}'$ and
\[c_*(A_j')=c_*(A_j),\quad c_*(A')=c_*(A).\]
Then
\[\widehat{A}_j:=A_j'\cup(A'\cap U_j)\subset A_j,\ j\in\mathbb N,\]
form an increasing sequence of $K_\sigma$-sets, whose union
\[\widehat{A}:=\bigcup_{j\in\mathbb N}\,\widehat{A}_j\supset A'\] is a $K_\sigma$-subset of $A$. The $K_\sigma$-sets being universally measurable,
\begin{equation*}\label{eq-cont'}c_*(A)=c_*(\widehat{A})=\lim_{j\to\infty}\,c_*(\widehat{A}_j)=
\lim_{j\to\infty}\,c_*(A_j),\end{equation*}
whence (\ref{eq-cont}). The second equality here holds by \cite[Lemma~2.3.3]{F1}, whereas the first (resp.\ the third) holds by (\ref{eq2}) with $H:=\widehat{A}$ (resp.\ with $A:=A_j$ and $H:=\widehat{A}_j$).

If now $c_*(A)$ is finite, then so is $c_*(\widehat{A})$, and \cite[Theorem~4.2]{F1} (applied to the universally measurable sets $\widehat{A}_j$, $j\in\mathbb N$, and $\widehat{A}$) gives
\[\gamma_{\widehat{A}_j}\to\gamma_{\widehat{A}}\text{ \ strongly and vaguely},\]
the kernel $\kappa$ being perfect. This proves (\ref{eq-cont1}), for, by Theorem~\ref{th-count},
\[\gamma_{\widehat{A}_j}=\gamma_{A_j},\quad\gamma_{\widehat{A}}=\gamma_A.\]

In view of the strong convergence of $(\gamma_{A_j})$ to $\gamma_A$, \cite[Lemma~3.2.4]{F1} yields
\[\kappa\gamma_A\geqslant\liminf_{j\to\infty}\,\kappa\gamma_{A_j}\text{ \ n.e.\ on $X$}.\]
The opposite being valid everywhere on $X$ by the vague convergence of $(\gamma_{A_j})$ to $\gamma_A$ (the principle of descent), (\ref{eq-cont2}) follows.

If the first and the second maximum principles are both fulfilled, the monotonicity property (Corollary~\ref{eq-mon}) implies that $(\kappa\gamma_{A_j})$ increases pointwise on $X$, and moreover \[\lim_{j\to\infty}\,\kappa\gamma_{A_j}\leqslant\kappa\gamma_A\text{ \ on $X$}.\] The opposite being obvious by the principle of descent (cf.\ above), this gives (\ref{eq-cont3}).

\subsection{Proof of Theorem~\ref{th-cont-bor2}}\label{pr-gap} As the net $(\|\gamma_{A_t}\|)$ is decreasing, we see from (\ref{eq.1.11}) that $(\gamma_{A_t})$ is strong Cauchy in $\mathcal E^+$, and there is therefore the unique $\gamma\in\mathcal E^+$ such that
\begin{equation}\label{eqcont}\gamma_{A_t}\to\gamma\text{ \ strongly and vaguely},\end{equation}
whence
\begin{equation}\label{bigcap1}\|\gamma\|^2=\lim_{t\in T}\,\|\gamma_{A_t}\|^2=\lim_{t\in T}\,c_*(A_t)\geqslant c_*(A).\end{equation}
Since $\kappa\gamma_{A_t}\leqslant1$ on $S(\gamma_{A_t})$ for every $t$, cf.\ (\ref{pr2}), we infer from the vague convergence of $(\gamma_{A_t})_{t\in T}$ to $\gamma$ that $\kappa\gamma\leqslant1$ on $S(\gamma)$, hence $\|\gamma\|^2=\int\kappa\gamma\,d\gamma\leqslant\gamma(X)$, and consequently
\begin{equation}\label{bigcap11}G(\gamma)=2\gamma(X)-\|\gamma\|^2\geqslant\|\gamma\|^2.\end{equation}

Suppose first that $A$ is the intersection of a decreasing net $(A_t)_{t\in T}$ of closed sets.
Then $\gamma$ belongs to the class $\mathcal E^+_{A_t}$ for every $t\in T$, $\mathcal E^+_{A_t}$ being strongly closed. Since for a closed set $F\subset X$, $\mathcal E^+_F$ consists of all $\nu\in\mathcal E^+$ supported by $F$, $\gamma$ is supported by every $A_t$, and hence by the intersection of $A_t$ over all $t$. Thus
\begin{equation}\label{bigcap}\gamma\in\mathcal E^+_A,\end{equation}
which in view of (\ref{G0}) yields
\[G(\gamma)\leqslant\sup_{\nu\in\mathcal E^+_A}\,G(\nu)=c_*(A).\]
Combining this with (\ref{bigcap1}) and (\ref{bigcap11}) we obtain
\[c_*(A)\geqslant G(\gamma)\geqslant\|\gamma\|^2\geqslant c_*(A),\]
whence $\gamma=\gamma_A$, by the uniqueness of the solution to extremal problem (\ref{G0}) (see Theorem~\ref{cor3}). Substituting $\gamma=\gamma_A$ into (\ref{eqcont}) gives (\ref{eq-cl}).

Let in addition the first and the second maximum principles be both fulfilled. By Corollary~\ref{eq-mon}, then the net $(\kappa\gamma_{A_t})_{t\in T}$ decreases pointwise on $X$, and moreover
\begin{equation}\label{ppp}\kappa\gamma_A(x)\leqslant\lim_{t\in T}\,\kappa\gamma_{A_t}(x)\text{ \ for all $x\in X$}.\end{equation}
The strong topology on $\mathcal E$ having a countable base of neighborhoods, it follows from (\ref{eq-cl}) that there exists a subsequence $(\gamma_{A_{t_j}})_{j\in\mathbb N}$ of the net $(\gamma_{A_t})_{t\in T}$ that converges strongly to $\gamma_A$. Applying \cite[Lemma~3.2.4]{F1} we therefore conclude that equality in fact prevails in (\ref{ppp}) for nearly all $x\in X$, which establishes (\ref{eq-cl-pot}).

Let $A$ now be the intersection of a decreasing sequence $(A_t)_{t\in T}$ of quasiclosed sets.
In view of the fact that a countable intersection of quasiclosed sets is likewise quasiclosed
\cite[Lemma~2.3]{F71}, the proof of (\ref{eq-cl}) and (\ref{eq-cl-pot}) is essentially the same as above, the only difference being in that of (\ref{bigcap}). Noting that for each $t$, the cone $\mathcal E^+_{A_t}$ is strongly closed (Lemma~\ref{l-quasi}), we obtain $\gamma\in\mathcal E^+_{A_t}$ (cf.\ above), which means that $(A_t)^c$ is locally $\gamma$-neg\-lig\-ible.
Being thus a countable union of locally $\gamma$-neg\-lig\-ible sets, $A^c$ is likewise locally $\gamma$-neg\-lig\-ible \cite[Section~IV.5.2]{B2}, which proves (\ref{bigcap}).

\section{Characterizations of $c_*(A)$ and $\gamma_A$ for arbitrary $A$}\label{sec-cr}

As in Sects.~\ref{sec-2}--\ref{sec-conv}, a kernel $\kappa$ is {\it perfect}. The purpose of this section is to establish alternative characterizations of the inner capacity $c_*(A)$ and the inner capacitary measure $\gamma_A$ for {\it arbitrary\/} sets $A\subset X$.
To achieve such a level of generality, assume in addition that {\it the first and the second maximum principles\/} are both fulfilled. Denote\footnote{\,$\widehat{\mathcal E}'_A$ would be the same if $\kappa\nu\leqslant1$ on $X$ were replaced by the apparently weaker inequality $\kappa\nu\leqslant1$ $\nu$-a.e., which is obvious from the lower semicontinuity of the potential $\kappa\nu$ on $X$, $\nu$ being positive, and Frostman's maximum principle.}
\begin{equation*}\label{QQ}\widehat{\mathcal E}'_A:=\bigl\{\nu\in\mathcal E'_A: \ \kappa\nu\leqslant1\text{\ on $X$}\bigr\},\end{equation*}
$\mathcal E'_A$ being the closure of $\mathcal E^+_A$ in the strong topology on $\mathcal E^+$.

\begin{theorem}[\footnote{Compare with Theorems~\ref{th-ap}, \ref{cor52}, \ref{cor3} and Corollary~\ref{cor52'}.}]\label{G0'ar}For arbitrary\/ $A\subset X$,
\begin{align}\label{G0ar}c_*(A)&=\sup_{\nu\in\mathcal E'_A}\,G(\nu),\\
\label{G00ar}c_*(A)&=\sup_{\nu\in\widehat{\mathcal E}'_A}\,\nu(X).\end{align}
If now\/ $c_*(A)<\infty$, the inner equilibrium measure\/ $\gamma_A$ can alternatively be characterized as the unique solution to either of extremal problems\/ {\rm(\ref{G0ar})} or\/ {\rm(\ref{G00ar})}, that is, $\gamma_A\in\widehat{\mathcal E}'_A\subset\mathcal E'_A$ and
\begin{align}
    &\max_{\nu\in\mathcal E'_A}\,G(\nu)=G(\gamma_A)\quad\bigl({}=c_*(A)\bigr),\label{cl1ar}\\
    &\max_{\nu\in\widehat{\mathcal E}'_A}\,\nu(X)=\gamma_A(X)\quad\bigl({}=c_*(A)\bigr).\label{cl2ar}
  \end{align}
Furthermore, $\gamma_A$ is uniquely characterized within\/ $\mathcal E'_A$ by each of\/ {\rm(a)} or\/ {\rm(c)}, where
\begin{itemize}
\item[{\rm(a)}] $\gamma_A(X)\geqslant\|\gamma_A\|^2\geqslant c_*(A)$.
\item[{\rm(c)}] $\kappa\gamma_A=1$ n.e.\ on\/ $A$.
\end{itemize}
\end{theorem}

\begin{remark}As noted in Remark~\ref{closedness}, in the case where $\mathcal E^+_A$ is strongly unclosed, problems (\ref{G0}) and (\ref{G00}) of maximizing $G(\nu)$, resp.\ $\nu(X)$, over $\mathcal E^+_A$, resp.\ $\widehat{\mathcal E}^+_A$, are in general {\it unsolvable}. To overcome this inconvenience, we have now enlarged the classes of admissible measures so that {\it the extremal values remain the same}, but {\it the extremal problems be already solvable}~--- even for arbitrary $A$. See (\ref{G0ar})--(\ref{cl2ar}).
\end{remark}

\begin{remark}Being valid for arbitrary $A\subset X$, Theorem~\ref{G0'ar} seems to be new even for the Green kernels associated with the Laplacian on Greenian sets in $\mathbb R^n$, $n\geqslant2$ (in particular, for the Newtonian kernel on $\mathbb R^n$, $n\geqslant3$), as well as for the $\alpha$-Riesz $|x-y|^{\alpha-n}$ and the associated $\alpha$-Green kernels on $\mathbb R^n$, where $\alpha<2\leqslant n$.\end{remark}

\begin{theorem}\label{th-char} For any\/ $A\subset X$ with\/ $c_*(A)<\infty$, the inner equilibrium measure\/ $\gamma_A$ can alternatively be characterized as the unique measure of minimum potential
in the class\/ $\Gamma^+_A$, that is, $\gamma_A\in\Gamma^+_A$ and\/\footnote{Recall that $\gamma_A$ is the unique measure of minimum {\it energy\/} in $\Gamma^+_A$ (Theorem~\ref{prop.1.2'}), and this holds true even for an arbitrary perfect kernel.}
\begin{equation}
\label{eq-char-p}\kappa\gamma_A=\min_{\nu\in\Gamma_A^+}\,\kappa\nu\text{ \ on\/ $X$.}\end{equation}
If moreover\/ $X$ is\/ $\sigma$-compact, then\/ $\gamma_A$ is also of minimum total mass in\/ $\Gamma^+_A$, that is,
\begin{equation}\label{eq-char-m}\gamma_A(X)=\min_{\nu\in\Gamma_A^+}\,\nu(X).\end{equation}
\end{theorem}

\begin{corollary}\label{cor-char}If\/ $X$ is\/ $\sigma$-compact, then for any\/ $A\subset X$,
\begin{equation}c_*(A)=\inf_{\nu\in\Gamma^+_A}\,\nu(X).\label{ch-i-cap'00}\end{equation}
\end{corollary}

\begin{proof}
  If $c_*(A)=\infty$, then $\Gamma^+_A=\varnothing$ (footnote~\ref{f5}), and the right-hand side in (\ref{ch-i-cap'00}) is $+\infty$ by convention. In the remaining case $c_*(A)<\infty$, (\ref{ch-i-cap'00}) is implied by (\ref{eq-char-m}).
\end{proof}

\begin{theorem}\label{cor2alt}For any\/ $A\subset X$ with\/ $c_*(A)<\infty$, $\gamma_A$ can alternatively be characterized as the unique measure in\/ $\mathcal E^+$ satisfying any of the three limit relations
\begin{align}&\gamma_K\to\gamma_A\text{ \ strongly in\/ $\mathcal E^+$ as\/ $K\uparrow A$},\notag\\
&\gamma_K\to\gamma_A\text{ \ vaguely in\/ $\mathcal E^+$ as\/ $K\uparrow A$},\label{apr1}\\
&\kappa\gamma_K\uparrow\kappa\gamma_A\text{ \ pointwise on\/ $X$  as\/ $K\uparrow A$},\label{apr2}
\end{align}
where\/ $\gamma_K$ denotes the only measure in\/ $\mathcal E^+_K$ having the property\/ $\kappa\gamma_K=1$ n.e.\ on\/ $K$.\end{theorem}

\begin{remark}
In $\mathbb R^n$, $n\geqslant2$, consider the $\alpha$-Riesz kernel $|x-y|^{\alpha-n}$ of order $\alpha\in(0,2]$, $\alpha<n$. Then the concept of inner equilibrium measure as well as its alternative characterizations, provided by formulae (\ref{eq-char-p}), (\ref{apr1}), and (\ref{apr2}), can be generalized to any $A\subset\mathbb R^n$ that is {\it inner\/ $\alpha$-thin at infinity}~--- even if the inner $\alpha$-Riesz capacity of $A$ is {\it infinite\/} (see \cite[Section~5]{Z-bal} and \cite[Sections~1.3, 2.1, 2.3]{Z-bal2}). Note that in this generalization, the inner $\alpha$-Riesz equilibrium measure $\gamma_A$ as well as the admissible measure in (\ref{eq-char-p}) must be allowed to be of {\it infinite\/} energy.
 \end{remark}

\subsection{Interaction between the concepts of equilibrium and swept measures} The proofs of Theorems~\ref{G0'ar}, \ref{th-char}, \ref{cor2alt} (Sects.~\ref{ch-pr3}--\ref{ch-pr2})
are based on the close interaction between the concepts of inner equilibrium and inner swept measures, discovered in Lemma~\ref{prop.1.1}, as well as on the results on inner balayage established in Sect.~\ref{i1}.

\begin{lemma}\label{prop.1.1}For any\/ $A,Q\subset X$ such that\/ $A\subset Q$ and\/ $c_*(Q)<\infty$,
\begin{equation}\label{eq.1.1}(\gamma_Q)^A=\gamma_A.\end{equation}
\end{lemma}

\begin{proof}By Definition~\ref{def-bal-n} and Lemma~\ref{def-bal-n-u} with $\mu:=\gamma_Q$, $(\gamma_Q)^A$ is the unique measure minimizing the energy over $\Gamma_{A,\gamma_Q}^+$, whereas $\gamma_A$ is the unique solution to the minimum energy problem over $\Gamma_A^+$ (Theorem~\ref{prop.1.2'}). It is thus enough to show that
\begin{equation}\label{GG}\Gamma_{A,\gamma_Q}^+=\Gamma_A^+.\end{equation}
This however is obvious in view of the fact that for any given $\nu\in\Gamma_{A,\gamma_Q}^+\cup\Gamma_A^+$,
\[\kappa\nu\geqslant1=\kappa\gamma_Q\text{ \ n.e.\ on $A$},\]
which in turn is derived from relations (\ref{io-n}), (\ref{F}), and (\ref{def-in}) (the first two being applied to $\gamma_Q$) by making use of Lemma~\ref{str}.
\end{proof}

\subsection{Proof of Theorem~\ref{G0'ar}}\label{ch-pr3}
Assume $c_*(A)<\infty$, for if not, (\ref{G0ar}), resp.\ (\ref{G00ar}), holds in the sense that its both sides equal $+\infty$, which is obvious from (\ref{G0}), resp.\ (\ref{G00}), in view of the inclusion $\mathcal E^+_A\subset\mathcal E'_A$, resp.\ $\widehat{\mathcal E}^+_A\subset\widehat{\mathcal E}'_A$.
Also assume that $c_*(A)>0$, for if not, $\mathcal E'_A=\widehat{\mathcal E}^+_A=\{0\}$ by Lemma~\ref{l-negl1}, and the theorem holds trivially with $\gamma_A=0$.

For each $\nu_0\in\mathcal E'_A$, there is a sequence $(\nu_j)\subset\mathcal E^+_A$ converging to $\nu_0$ both strongly and vaguely, the kernel $\kappa$ being perfect. Noting that $\nu(X)$ is {\it vaguely lower semicontinuous\/} on positive measures, whereas $\|\nu\|^2$ is {\it strongly continuous\/} on $\mathcal E^+$, we obtain
\[G(\nu_0)\leqslant\liminf_{j\to\infty}\,G(\nu_j)\leqslant\sup_{\nu\in\mathcal E^+_A}\,G(\nu),\]
and letting now $\nu_0$ range over $\mathcal E'_A$ gives
\[\sup_{\nu\in\mathcal E'_A}\,G(\nu)\leqslant\sup_{\nu\in\mathcal E^+_A}\,G(\nu).\]
The opposite being obvious, combining this with (\ref{G0}) establishes (\ref{G0ar}).

The class $\mathcal E'_A$ being strongly closed, $\gamma_A\in\mathcal E'_A$ (cf.\ Lemma~\ref{O} or Theorem~\ref{cor2}). Furthermore, $\gamma_A\in\widehat{\mathcal E}'_A$, for $\kappa\gamma_A\leqslant1$ on $X$ by Frostman's maximum principle.

Since according to Lemma~\ref{prop.1.1} with $Q:=A$,
\begin{equation}\label{AA}\gamma_A=(\gamma_A)^A,\end{equation}
we infer from Theorem~\ref{th-intr}(b) that
\begin{equation*}\label{gammaa'}\min_{\nu\in\mathcal E_A'}\,\|\gamma_A-\nu\|^2=0,\end{equation*}
the minimum being attained at the unique $\gamma_A\in\mathcal E_A'$.\footnote{The inner equilibrium measure $\gamma_A$ serves here as the (unique) orthogonal projection of $\gamma_A$ in the pre-Hil\-bert space $\mathcal E$ onto the convex, strongly complete cone $\mathcal E'_A$.}
Consequently,
\begin{equation}\label{gammaa}c_*(A)=\|\gamma_A\|^2=
\max_{\nu\in\mathcal E_A'}\,\Bigl(2\int\kappa\gamma_A\,d\nu-\|\nu\|^2\Bigr),\end{equation}
the maximum being likewise attained at the same (unique) $\gamma_A\in\mathcal E_A'$.

On account of the fact that $\kappa\gamma_A\leqslant1$ on $X$, we therefore obtain
\[c_*(A)=\max_{\nu\in\mathcal E_A'}\,\Bigl(2\int\kappa\gamma_A\,d\nu-\|\nu\|^2\Bigr)\leqslant\sup_{\nu\in\mathcal E'_A}\,\bigl(2\nu(X)-\|\nu\|^2\bigr)=\sup_{\nu\in\mathcal E'_A}\,G(\nu)=c_*(A),\]
the last equality being valid according to (\ref{G0ar}). As $G(\gamma_A)=c_*(A)$, this implies that (\ref{cl1ar}) indeed holds true with the unique maximizing measure $\gamma_A$, the uniqueness being clear from the uniqueness of the maximizing measure in (\ref{gammaa}).

For every $\nu\in\widehat{\mathcal E}'_A$, we have $\|\nu\|^2\leqslant\nu(X)$, hence $G(\nu)\geqslant\nu(X)$, and consequently
\begin{equation}\label{prpr}c_*(A)=\gamma_A(X)\leqslant\sup_{\nu\in\widehat{\mathcal E}'_A}\,\nu(X)\leqslant\sup_{\nu\in\widehat{\mathcal E}'_A}\,G(\nu)\leqslant\max_{\nu\in\mathcal E'_A}\,G(\nu)=G(\gamma_A)=c_*(A),\end{equation}
the first inequality being valid by $\gamma_A\in\widehat{\mathcal E}'_A$, and the last two equalities by (\ref{cl1ar}). This proves both (\ref{G00ar}) and (\ref{cl2ar}). Furthermore, the solution $\gamma_A$ to problem (\ref{G00ar}) is unique, which is seen from (\ref{prpr}) in view of the uniqueness of the solution to problem (\ref{G0ar}).

According to (\ref{pr1}) and (\ref{F}), $\gamma_A$ has properties (a) and (c). It thus remains to show that each of these (a) and (c) determines $\gamma_A$ uniquely within $\mathcal E'_A$.

If $\mu$ is another measure in $\mathcal E'_A$ satisfying (a), then,  by (\ref{cl1ar}),
\[c_*(A)=\max_{\nu\in\mathcal E'_A}\,G(\nu)\geqslant G(\mu)=2\mu(X)-\|\mu\|^2\geqslant\|\mu\|^2\geqslant c_*(A).\]
Thus $\mu$ must be the (unique) solution to problem (\ref{G0ar}), and so indeed $\mu=\gamma_A$.

Assume finally that $\tau$ is another measure in $\mathcal E'_A$ with $\kappa\tau=1$ n.e.\ on $A$; then
\begin{equation*}\label{cc}\kappa\tau=\kappa\gamma_A\text{ \ n.e.\ on $A$},\end{equation*}
which is clear from (\ref{F}) by use of Lemma~\ref{str}. But according to Theorem~\ref{th-intr}(c) with $\mu:=\gamma_A$, $(\gamma_A)^A$ is the only measure in $\mathcal E'_A$ whose potential equals $\kappa\gamma_A$ nearly everywhere on $A$. Hence, $\tau=(\gamma_A)^A$, which combined with (\ref{AA}) gives $\tau=\gamma_A$.

\subsection{Proof of Theorem~\ref{th-char}}\label{ch-pr2''} On account of (\ref{AA}), we infer from Theorem~\ref{th-intr}(a) with $\mu:=\gamma_A$ that $\gamma_A$ is the only measure of minimum potential in the class $\Gamma_{A,\gamma_A}$:
\begin{equation*}\label{comp}\kappa\gamma_A=\min_{\nu\in\Gamma_{A,\gamma_A}}\,\kappa\nu\text{ \ on $X$},\end{equation*}
which combined with (\ref{GG}) with $Q:=A$ results in (\ref{eq-char-p}). Similarly, (\ref{eq-char-m}) follows by substituting (\ref{GG}) (with $Q:=A$) and (\ref{AA}) into (\ref{eq-t-m}) with $\mu:=\gamma_A$.

\subsection{Proof of Theorem~\ref{cor2alt}}\label{ch-pr2} In view of (\ref{AA}), we derive this from Theorem~\ref{th-intr}(d) with $\mu:=\gamma_A$ noting that the unique $\nu\in\mathcal E^+_K$ with  $\kappa\nu=\kappa\gamma_A$ n.e.\ on $K$ obviously coincides with the unique measure in $\mathcal E^+_K$ whose potential equals $1$ n.e.\ on $K$.

\section{On convergence and alternative characterizations of outer capacities and outer capacitary measures}\label{sec-ou-chara}

Assuming as usual that a kernel $\kappa$ is {\it perfect}, we shall now make the following two additional hypotheses concerning a locally compact space $X$:
\begin{itemize}\item[(H$_1$)] $X$ {\it is perfectly normal}.\footnote{By Urysohn's theorem \cite[Section~IX.1, Theorem~1]{B3}, a Hausdorff space $Y$ is said to be {\it normal\/} if for any disjoint closed $F_1,F_2\subset Y$, there exist disjoint open $D_1,D_2$ such that $F_i\subset D_i$ $(i=1,2)$. A normal space $Y$ is said to be {\it perfectly normal\/} \cite[Section~IX.4, Exercise~7]{B3} if each closed subset of $Y$ is a countable intersection of open sets.}
\item[(H$_2$)] $X$ {\it is\/ $\sigma$-compact}.\end{itemize}

\begin{remark} Given a locally compact space $X$, {\it both\/ {\rm(H$_1$)} and\/ {\rm(H$_2$)} are fulfilled if $X$ is sec\-ond-count\-able, but not the other way around.} Indeed, according to \cite[Section~IX.2, Corollary to Proposition~16]{B3}, $X$ is sec\-ond-count\-able if and only if it is metrizable and $\sigma$-com\-pact. Being therefore metrizable, a sec\-ond-count\-able, locally compact space $X$ must be perfectly normal \cite[Section~IX.1, Proposition~2]{B3}; whereas the converse fails in general by virtue of \cite[Section~IX.2, Exercise~13(b)]{B3}.\end{remark}

\begin{theorem}\label{l-top}Any Borel subset of a locally compact space\/ $X$ satisfying hypotheses\/ {\rm(H$_1$)} and\/ {\rm(H$_2$)} and endowed with a perfect kernel\/ $\kappa$, is capacitable.\end{theorem}

\begin{proof}This follows directly from \cite[Theorem~4.5]{F1}.\end{proof}

In the remainder of this section, assume $A\subset X$ to be {\it Borel}, hence capacitable:
\[c_*(A)=c^*(A)\quad\bigl({}=:c(A)\bigr).\]
If moreover $c(A)<\infty$, then the (unique) inner capacitary measure $\gamma_A$ (Theorem~\ref{prop.1.2'}) serves simultaneously as the {\it outer capacitary measure\/} $\gamma^*_A$, defined as the (unique) measure of minimum energy  in the class $\Gamma^*_A$, where
\[\Gamma^*_A:=\bigl\{\nu\in\mathcal E^+: \ \kappa\nu\geqslant1\text{ \ q.e.\ on $A$}\bigr\}.\]
Indeed, the equality $\gamma_A=\gamma_A^*$ can be concluded from \cite[Theorem~4.3]{F1} noting that
\[\Gamma^+_A=\Gamma^*_A,\]
which in turn follows from the fact that $A\cap\{\kappa\nu<1\bigr\}$ is Borel, hence capacitable.

This makes it possible to utilize the results on inner capacities and inner capacitary measures, obtained in Sect.~\ref{sec-2}--\ref{sec-cr} above, in order to perform a similar analysis for outer capacities and outer capacitary measures. We are thus led to the following Theorems~\ref{cor2-ou}--\ref{cor2alt-ou} (compare with Theorems~\ref{th-ap}, \ref{cor52}, \ref{cor3}, \ref{cor2}, \ref{th-cont-bor2}, \ref{G0'ar},  \ref{th-char}, \ref{cor2alt}).

\begin{theorem}\label{cor2-ou}For any Borel\/ $A\subset X$ with\/ $c^*(A)<\infty$,
\[\gamma_K\to\gamma_A^*\text{ \ strongly and vaguely in\/ $\mathcal E^+$ as\/ $K\uparrow A$.}\]
If moreover the first and the second maximum principles both hold, then also
\[\kappa\gamma_K\uparrow\kappa\gamma_A^*\text{ \ pointwise on\/ $X$ as\/ $K\uparrow A$.}\]
\end{theorem}

\begin{theorem}\label{th-cont-bor2-ou}Let\/ $A$ be the intersection of a decreasing sequence\/ $(A_j)$ of Borel, quasiclosed sets\/ $A_j\subset X$ such that\/
$c^*(A_{j_0})<\infty$ for some\/ $j_0$. Then
\[\gamma^*_{A_j}\to\gamma^*_A\text{ \ strongly and vaguely as\/ $j\to\infty$}.\]
If moreover the first and the second maximum principles both hold, then also
\[\kappa\gamma^*_{A_j}\downarrow\kappa\gamma^*_A\text{ \ pointwise q.e.\ on\/ $X$}.\]
\end{theorem}

\begin{theorem}\label{G0'-ou} For any Borel\/ $A\subset X$,
\begin{equation}\label{G0-ou}c^*(A)=\sup_{\nu\in\mathcal E^+_A}\,G(\nu)=
\sup_{\nu\in\widehat{\mathcal E}^+_A}\,\nu(X).\end{equation}
If moreover the class\/ $\mathcal E^+_A$ is strongly closed\/ {\rm(}or in particular if\/ $A$ is quasiclosed\/{\rm)}, and if\/ $c^*(A)<\infty$, then each of the two extremal problems appearing in\/ {\rm(\ref{G0-ou})} has the same unique solution, and it exactly equals the outer capacitary measure\/ $\gamma_A^*$.
This\/ $\gamma_A^*$ can also be uniquely characterized within\/ $\mathcal E^+_A$ by either of\/ {\rm(a$'$)} or\/ {\rm(b$'$)}, where
\begin{itemize}
\item[{\rm(a$'$)}] $\gamma_A^*(X)\geqslant\|\gamma_A^*\|^2\geqslant c^*(A)$.
\item[{\rm(b$'$)}] $\kappa\gamma_A^*\geqslant1$ q.e.\ on $A$, and\/ $\kappa\gamma_A^*\leqslant1$ $\gamma_A^*$-a.e.
\end{itemize}
If in addition Frostman's maximum principle is fulfilled, then\/ $\gamma_A^*$ is the only measure in\/ $\mathcal E^+_A$ having the property
\begin{itemize}
\item[{\rm(c$'$)}] $\kappa\gamma_A^*=1$ q.e.\ on\/ $A$.
\end{itemize}
\end{theorem}

In the following three theorems assume additionally that the (perfect) kernel $\kappa$ satisfies both {\it the first and the second maximum principles}.

\begin{theorem}\label{char-ou1} For any Borel\/ $A\subset X$,
\begin{equation}
c^*(A)=\sup_{\nu\in\mathcal E'_A}\,G(\nu)=\sup_{\nu\in\widehat{\mathcal E}'_A}\,\nu(X).\label{ch-ou5-ou}
\end{equation}
If now\/ $c^*(A)<\infty$, then each of the two extremal problems appearing in\/ {\rm(\ref{ch-ou5-ou})} has the same unique solution, and it exactly equals the outer equilibrium measure\/ $\gamma_A^*$. This\/ $\gamma_A^*$ can also be uniquely characterized within\/ $\mathcal E'_A$ by either of\/ {\rm(a$'$)} or\/ {\rm(c$'$)}, where
\begin{itemize}
\item[{\rm(a$'$)}] $\gamma_A^*(X)\geqslant\|\gamma_A^*\|^2\geqslant c^*(A)$.
\item[{\rm(c$'$)}] $\kappa\gamma^*_A=1$ q.e.\ on\/ $A$.
\end{itemize}
\end{theorem}

\begin{theorem}\label{th-char-ou} For any Borel\/ $A\subset X$ with\/ $c^*(A)<\infty$,
the outer equilibrium measure\/ $\gamma^*_A$ can alternatively be characterized as the unique measure of minimum potential in the class\/ $\Gamma^*_A$, that is,
\[\kappa\gamma^*_A=\min_{\nu\in\Gamma_A^*}\,\kappa\nu\text{ \ on\/ $X$.}\]
Furthermore, $\gamma^*_A$ is also of minimum total mass in the class\/ $\Gamma^*_A$, that is,
\[\gamma_A^*(X)=\min_{\nu\in\Gamma_A^*}\,\nu(X).\]
\end{theorem}

\begin{theorem}\label{cor2alt-ou}For any Borel\/ $A\subset X$ with\/ $c^*(A)<\infty$, $\gamma^*_A$ can alternatively be characterized as the unique measure in\/ $\mathcal E^+$ satisfying any of the three limit relations
\begin{align*}&\gamma_K^*\to\gamma^*_A\text{ \ strongly in\/ $\mathcal E^+$ as\/ $K\uparrow A$},\\
&\gamma_K^*\to\gamma^*_A\text{ \ vaguely in\/ $\mathcal E^+$ as\/ $K\uparrow A$},\\
&\kappa\gamma_K^*\uparrow\kappa\gamma^*_A\text{ \ pointwise on\/ $X$  as\/ $K\uparrow A$},
\end{align*}
where\/ $\gamma_K^*$ denotes the only measure in\/ $\mathcal E^+_K$ having the property\/ $\kappa\gamma_K^*=1$ q.e.\ on\/ $K$.\end{theorem}

\section{Alternative characterizations of outer swept measures}\label{sec-bal-ou}

In this section we assume that a locally compact space $X$ is {\it $\sigma$-com\-pact and perfectly normal}, and that a kernel $\kappa$ is {\it perfect and satisfies the domination principle}.

For $\mu\in\mathcal E^+$ and $A\subset X$, denote
\begin{equation}\label{eq-bal-ouu}\Gamma_{A,\mu}^*:=\bigl\{\nu\in\mathcal E^+: \ \kappa\nu\geqslant\kappa\mu\text{ \ q.e.\ on $A$}\bigr\}.\end{equation}

\begin{definition}\label{def-ou-ba}The {\it outer balayage\/} $\mu^{*A}$ of $\mu\in\mathcal E^+$ to $A\subset X$ is defined as the measure of minimum energy in the class $\Gamma_{A,\mu}^*$, that is,  $\mu^{*A}\in\Gamma_{A,\mu}^*$ and
\begin{equation*}\|\mu^{*A}\|^2=\min_{\nu\in\Gamma_{A,\mu}^*}\,\|\nu\|^2.\end{equation*}\end{definition}

Observe that this definition differs from that of inner balayage (Definition~\ref{def-bal-n}) only by replacing an exceptional set in (\ref{io-n}) by that of {\it outer capacity\/} zero, cf.\ (\ref{eq-bal-ouu}).

\begin{theorem}\label{th-bal-ou}
  For any\/ $\mu\in\mathcal E^+$ and any Borel\/ $A\subset X$, there exists the unique outer balayage\/ $\mu^{*A}$, introduced by Definition\/~{\rm\ref{def-ou-ba}}, and it satisfies the three relations  \begin{align*}
\kappa\mu^{*A}&=\kappa\mu\text{ \ q.e.\ on\ }A,\\
\kappa\mu^{*A}&=\kappa\mu\text{ \ $\mu^{*A}$-a.e.,}\\
\kappa\mu^{*A}&\leqslant\kappa\mu\text{ \ on $X$.}
\end{align*}
Furthermore, the outer balayage\/ $\mu^{*A}$ actually coincides with the inner balayage\/ $\mu^A$,
and it can alternatively be characterized by means of any of the following\/ {\rm(}equivalent\/{\rm)} assertions\/ {\rm(a$'$)--(d$'$)}:
\begin{itemize}
\item[\rm(a$'$)] $\mu^{*A}$ is the unique measure of minimum potential in\/ $\Gamma_{A,\mu}^*$, that is, $\mu^{*A}\in\Gamma_{A,\mu}^*$ and
\[\kappa\mu^{*A}=\min_{\nu\in\Gamma_{A,\mu}^*}\,\kappa\nu\text{ \ on\/ $X$}.\]
\item[\rm(b$'$)] $\mu^{*A}$ is the unique orthogonal projection of\/ $\mu$ in the pre-Hilbert space\/ $\mathcal E$ onto the\/ {\rm(}convex, strongly complete\/{\rm)} cone $\mathcal E'_A$, that is, $\mu^{*A}\in\mathcal E'_A$ and
\[\|\mu-\mu^{*A}\|=\min_{\nu\in\mathcal E_A'}\,\|\mu-\nu\|.\]
\item[\rm(c$'$)] $\mu^{*A}$ is the unique measure in\/ $\mathcal E'_A$ having the property
\[\kappa\mu^{*A}=\kappa\mu\text{ \ q.e.\ on\ }A.\]
\item[\rm(d$'$)] $\mu^{*A}$ is the unique measure in\/ $\mathcal E^+$ satisfying any of the three limit relations
\begin{align*}
 &\mu^{*K}\to\mu^{*A}\text{ \ strongly in\/ $\mathcal E^+$ as\/ $K\uparrow A$},\\
 &\mu^{*K}\to\mu^{*A}\text{ \ vaguely in\/ $\mathcal E^+$ as\/ $K\uparrow A$},\\
 &\kappa\mu^{*K}\uparrow\kappa\mu^{*A}\text{ \ pointwise on\/ $X$ as\/ $K\uparrow A$},
\end{align*}
where\/ $\mu^{*K}$ denotes the only measure in\/ $\mathcal E^+_K$ with\/ $\kappa\mu^{*K}=\kappa\mu$ q.e.\ on\/ $K$.
\end{itemize}
If moreover the class\/ $\mathcal E^+_A$ is strongly closed\/ {\rm(}or in particular if the set\/ $A$ is quasiclosed\/{\rm)}, then\/ {\rm(b$'$)} and\/ {\rm(c$'$)} remain valid with\/ $\mathcal E^+_A$ in place of\/ $\mathcal E'_A$.
\end{theorem}

\begin{proof} We first observe that the classes $\Gamma_{A,\mu}^+$ and $\Gamma_{A,\mu}^*$, introduced by (\ref{io-n}) and (\ref{eq-bal-ouu}), respectively, coincide:
\begin{equation}\label{theta}\Gamma_{A,\mu}^+=\Gamma_{A,\mu}^*,\end{equation}
which follows from the fact that for any $\mu,\nu\in\mathcal E^+$, the set $A\cap\{\kappa\mu\ne\kappa\nu\}$ is Borel, hence capacitable (Theorem~\ref{l-top}).
This in turn implies that the inner balayage $\mu^A$ serves simultaneously as the outer balayage $\mu^{*A}$, i.e.
\begin{equation}\label{inotb}
  \mu^{*A}=\mu^A,
\end{equation}
and
all we need to prove can therefore be immediately derived from Theorem~\ref{th-intr} and Lemma~\ref{l-quasi}.
\end{proof}

In the following two corollaries assume in addition that {\it Frostman's maximum principle\/} is fulfilled.

\begin{corollary}\label{bal-tot-m-ou}
For any\/ $\mu\in\mathcal E^+$ any any Borel\/ $A\subset X$, the outer balayage\/ $\mu^{*A}$ is of minimum total mass in the class\/ $\Gamma_{A,\mu}^*$, that is,
\[\mu^{*A}(X)=\min_{\nu\in\Gamma_{A,\mu}^*}\,\nu(X).\]
\end{corollary}

\begin{proof} This follows by substituting (\ref{theta}) and (\ref{inotb}) into (\ref{eq-t-m}).\end{proof}

\begin{corollary}\label{prop.1.1-ou}
For any Borel sets\/ $A,Q\subset X$ such that\/ $A\subset Q$ and\/ $c^*(Q)<\infty$,
\[(\gamma^*_Q)^{*A}=\gamma^*_A.\]
\end{corollary}

\begin{proof} As observed in Sect.~\ref{sec-ou-chara}, $\gamma_A=\gamma_A^*$ and $\gamma_Q=\gamma_Q^*$. Substituting this and (\ref{inotb}) with $\mu:=\gamma^*_Q$ into (\ref{eq.1.1}) establishes the corollary.\end{proof}

\section{Acknowledgements} The author is deeply indebted to Bent Fuglede for reading and commenting on the manuscript.

\end{document}